\documentclass[11pt,english,british]{extarticle}
\usepackage[T1]{fontenc}
\usepackage[utf8]{inputenc}

\usepackage{amsmath,amssymb,amsthm}

\usepackage{graphicx}
\usepackage{booktabs}
\usepackage{array}
\usepackage{float}
\usepackage{subcaption}
\usepackage{caption}
\usepackage{adjustbox}
\usepackage{tabularx}
\usepackage{makecell}

\usepackage[numbers]{natbib}

\usepackage[a4paper,margin=1in]{geometry}

\newtheorem{thm}{Theorem}
\newtheorem{lem}{Lemma}

\newtheorem{remark}{Remark}

\usepackage[colorlinks=true,linkcolor=black,citecolor=blue,urlcolor=black]{hyperref}

\begin{document}
\title{A Layer-Resolving Rational Trial-Space Method for Convection-Dominated Convection–Diffusion Equations}
\author{
Zihao Guo$^{1}$ \and Xin Li$^{2,*}$ \and Zhihong Xia$^{3,\dagger}$
}
\date{
$^{1}$Institute for Advanced Research, Great Bay University, Dongguan, China\\
$^{2}$College of Computer Science and Technology, Dongguan University of Technology, Dongguan, China\\
$^{3}$Department of Mathematics, Northwestern University, Evanston, USA\\[6pt]
$^{*}$Corresponding author: \protect\href{mailto:xinli2023@u.northwestern.edu}{xinli2023@u.northwestern.edu}\\
$^{\dagger}$Co-corresponding author: \protect\href{mailto:xia@math.northwestern.edu}{xia@math.northwestern.edu}
}
\maketitle
\begin{abstract}
Convection-dominated convection--diffusion problems develop thin layers
whose values, gradients, and curvatures exhibit strongly different physical
scales. Standard neural trial spaces do not explicitly encode this coupled
structure. We introduce an integrated Cauchy trial space for resolving such
layers. Each ridge atom admits a geometric interpretation through its normal
direction, location, and physical width
\(\rho=d/\|\boldsymbol w\|\). The atom represents a transition at the
solution level, while its first and second derivatives generate localized
Cauchy-type gradient and curvature profiles.

For analytic layer profiles in stretched coordinates, we prove that the same
collection of integrated Cauchy atoms approximates the profile and its first
two derivatives with a common exponential rate. After rescaling to physical
space, this yields layer-thickness-uniform best-approximation estimates in
appropriately scaled norms, covering both \(O(\varepsilon)\) exponential
layers and \(O(\sqrt{\varepsilon})\) characteristic layers. The resulting
trial space is compatible with both strong and variational formulations.
Numerical experiments on interior, boundary, outflow, corner, and curved
layers show improved accuracy and parameter efficiency over standard neural
representations, and further improvements when the proposed trial space is
embedded into stabilized hp-VPINN frameworks.

\medskip{}


\medskip{}

\end{abstract}

\section{Introduction}
\label{sec:introduction}

Convection--diffusion equations are fundamental models for transport
phenomena in fluid mechanics, heat and mass transfer, plasma dynamics,
semiconductor devices, and porous media
\cite{evans2010,ferziger2002,bird2002,bear1972,depaoli2023}. A representative
steady problem is
\begin{equation}
    -\varepsilon\Delta u
    +
    \boldsymbol b\cdot\nabla u
    +
    cu
    =
    f
    \qquad
    \text{in }\Omega,
    \label{eq:intro_convection_diffusion}
\end{equation}
supplemented with suitable boundary conditions. Here,
\(\varepsilon>0\) is the diffusion coefficient,
\(\boldsymbol b\) is the convection field, \(c\) is a reaction
coefficient, and \(f\) is a source term. When
\(0<\varepsilon\ll1\), the problem is singularly perturbed and its
solutions may contain thin boundary, interior, outflow, characteristic,
corner, or curved layers \cite{roos2008,augustin2011}. These structures
occupy only a small fraction of the physical domain, but they frequently
determine the global accuracy and stability of the numerical solution.
Failure to resolve them can lead to excessive numerical diffusion,
spurious oscillations, violations of physical bounds, and large errors in
fluxes and differential residuals \cite{morton1996}.

The difficulty is not merely that the solution varies rapidly. Locally,
the leading variation across a layer of physical thickness \(\delta\) can
be represented in a stretched normal coordinate as
\begin{equation}
    u_\delta(\boldsymbol x)
    =
    U\left(
        \frac{\boldsymbol n\cdot\boldsymbol x-s}{\delta}
    \right),
    \qquad
    \|\boldsymbol n\|=1,
    \label{eq:intro_local_layer}
\end{equation}
where \(s\) specifies the layer location and
\(\boldsymbol n\) its normal direction. The corresponding differential
scaling is
\begin{equation}
    u_\delta=O(1),
    \qquad
    \nabla u_\delta=O(\delta^{-1}),
    \qquad
    D^2u_\delta=O(\delta^{-2}).
    \label{eq:intro_layer_scaling}
\end{equation}
Thus the reaction, convection, and diffusion terms act on three coupled
structures: an \(O(1)\) transition, a localized first derivative, and a
localized, often sign-changing curvature profile. An approximation that
is accurate only at the solution level may therefore remain inaccurate
after application of the differential operator.

Convection--diffusion problems can also contain several distinct layer
scales. If the convection field has a nonzero normal component at an
outflow boundary, the local balance between normal diffusion and
convection gives an exponential layer thickness of the form
\begin{equation}
    \delta_E
    \asymp
    \frac{\varepsilon}
         {|\boldsymbol b\cdot\boldsymbol n_E|},
    \label{eq:intro_exponential_layer_scale}
\end{equation}
which is \(O(\varepsilon)\) when the normal convection is of order one.
At a characteristic boundary, however,
\(\boldsymbol b\cdot\boldsymbol n_P=0\), and normal diffusion is balanced
by tangential transport or reaction. The resulting characteristic or
parabolic layer generally has thickness
\begin{equation}
    \delta_P
    \asymp
    \sqrt{\varepsilon}.
    \label{eq:intro_parabolic_layer_scale}
\end{equation}
A single solution may contain both scales, together with their interaction
near corners. Resolving such problems requires a trial space that can
adapt its physical scale without increasing its representation complexity
in proportion to \(\varepsilon^{-1}\).

Classical numerical methods address convection dominance through
stabilization, fitted or adaptive meshes, and layer-aware approximation
spaces. Representative approaches include streamline
upwind/Petrov--Galerkin methods \cite{brooks1982,giera2015},
Galerkin/least-squares schemes \cite{hughes1989gls}, variational multiscale
methods \cite{hughes1995vms,codina2002oss}, fitted finite difference and
finite element discretizations, discrete-maximum-principle-preserving
methods \cite{Barrenechea2024}, and layer-adapted or adaptive finite
elements \cite{linss2010}. Although these methods differ substantially,
they share an important principle: in the convection-dominated regime,
uniform approximation alone is generally insufficient. The mesh,
trial space, test space, or stabilization mechanism must reflect the
location, direction, or scale of the unresolved layer.

Physics-informed neural networks (PINNs) \cite{raissi2019} provide a
mesh-free alternative in which a neural trial function is determined by
minimizing residuals of the governing equation and boundary conditions.
Numerous extensions have been proposed, including adaptive sampling and
training strategies \cite{wang2024}, domain decomposition
\cite{kharazmi2020,jagtap2020}, Fourier-feature representations
\cite{li2023}, and KAN-based architectures
\cite{toscano2025,patra2025}. Physics-informed methods have been applied to
a broad range of forward, inverse, and multiphysics problems
\cite{Karniadakis2021,cai2021,haghighat2021,jin2021}. Nevertheless,
transport-dominated and singularly perturbed equations remain challenging
for standard neural architectures. Sharp layers produce highly
anisotropic functions and derivatives, which can lead to severe
approximation, sampling, and optimization difficulties
\cite{wang2022,krishnapriyan2021,wang2022curriculum}.

Recent work on neural methods for convection-dominated problems has
therefore focused on improved residual formulations, stabilization,
boundary treatment, and training procedures. Examples include strong- and
weak-form PINNs \cite{sikora2023}, bound-preserving formulations
\cite{matthaiou2025}, stabilized hp-VPINNs \cite{anandh2025}, and studies
of alternative loss functionals for singularly perturbed transport
problems \cite{frerichs2026}. Such developments are important because they
improve residual enforcement and suppress nonphysical behavior. However,
the approximation is still performed in the neural trial space generated
by the network. If this space does not efficiently represent the
transition and its localized derivatives, modifications of the loss
functional do not by themselves eliminate the underlying representation
difficulty.

The present work follows a complementary direction. Rather than changing
only the loss or stabilization strategy, we construct and analyze a neural
trial space designed around the differential structure of
convection--diffusion layers. The desired basis element should represent
an \(O(1)\) transition at the solution level, while its first and second
derivatives remain localized on the same physical scale and exhibit the
\(O(\delta^{-1})\) and \(O(\delta^{-2})\) amplitudes in
\eqref{eq:intro_layer_scaling}. It should also permit the physical width
of the basis element to scale with either
\(\delta_E=O(\varepsilon)\) or
\(\delta_P=O(\sqrt{\varepsilon})\), without requiring an increasing
number of basis functions solely because the layer has become thinner.

Cauchy rational representations provide a natural starting point for this
construction. XNet \cite{li2025,li20252} introduced real-valued Cauchy
representations motivated by the Cauchy integral formula and its discrete
quadrature, and demonstrated their approximation capability and parameter
efficiency. Their use in physics-informed computation was subsequently
investigated in \cite{si2026}. A Cauchy function is localized and is
therefore well suited to peak-type or derivative-type structures.
A convection--diffusion layer, however, is usually transition-type at the
solution level, with localization appearing only after differentiation.
This suggests integrating the Cauchy representation rather than applying
it directly to the solution values.

We therefore introduce an integrated Cauchy layer trial space. For an
integrated Cauchy ridge with affine parameters
\((\boldsymbol w,b,d)\), define
\begin{equation}
    \boldsymbol n
    =
    \frac{\boldsymbol w}{\|\boldsymbol w\|},
    \qquad
    s
    =
    -\frac{b}{\|\boldsymbol w\|},
    \qquad
    \rho
    =
    \frac{d}{\|\boldsymbol w\|}.
    \label{eq:intro_geometric_parameters}
\end{equation}
These quantities describe the normal direction, location, and physical
width of the corresponding ridge feature. In geometric form, an atom can
be written as a combination of arctangent and logarithmic components whose
normal coordinate is
\[
    \tau
    =
    \frac{\boldsymbol n\cdot\boldsymbol x-s}{\rho}.
\]
The atom is transition-type at the solution level, its first derivative
is a localized Cauchy profile proportional to \(\rho^{-1}\), and its
second derivative is a localized Cauchy-derivative profile proportional to
\(\rho^{-2}\). Consequently, one collection of atom parameters generates
the complete value--gradient--curvature hierarchy required by the
differential operator.

The main analytical result of this work is stronger than a value-only
approximation statement. Let \(U\) be an analytic layer profile in an
order-one stretched coordinate. We first approximate \(U'\) by a Cauchy
sum obtained from a discretized Cauchy integral. Integrating the same
Cauchy atoms gives an integrated Cauchy approximation \(G_M\) of \(U\),
while a complex-neighborhood Cauchy estimate controls the derivative of
the approximation error. This construction yields
\begin{equation}
    \sum_{k=0}^{2}
    \left\|
        U^{(k)}-G_M^{(k)}
    \right\|_{L^\infty(J)}
    \le
    C\sigma^{-M},
    \qquad
    \sigma>1,
    \label{eq:intro_derivative_stable_rate}
\end{equation}
for the same set of integrated Cauchy atoms. Thus the transition profile,
its localized gradient, and its curvature are approximated with a common
exponential rate.

After mapping the approximation back to physical space, the derivatives
acquire their natural factors of \(\delta^{-1}\) and \(\delta^{-2}\).
Accordingly, we prove the thickness-uniform estimate
\begin{equation}
\begin{aligned}
    &
    \|u_\delta-g_{M,\delta}\|_{L^\infty}
    +
    \delta
    \|\nabla u_\delta-\nabla g_{M,\delta}\|_{L^\infty}
    \\
    &\qquad
    +
    \delta^2
    \|D^2u_\delta-D^2g_{M,\delta}\|_{L^\infty}
    \le
    C\sigma^{-M},
\end{aligned}
    \label{eq:intro_physical_derivative_rate}
\end{equation}
together with an equivalent estimate in a layer-scaled \(H^2\)-type norm.
The constants are independent of the physical thickness \(\delta\),
provided the atom widths satisfy
\[
    \rho_m
    =
    \delta\widehat\rho_m
\]
for order-one stretched widths \(\widehat\rho_m\). In particular,
different groups of atoms can simultaneously resolve exponential and
characteristic layers by using
\[
    \rho_m^E=O(\varepsilon),
    \qquad
    \rho_m^P=O(\sqrt{\varepsilon}).
\]
For a fixed number of separated analytic layer components, the resulting
best-approximation complexity depends logarithmically on the target
accuracy and does not deteriorate with the minimum physical layer
thickness. This statement concerns representation complexity; it does not
assert that optimization iterations, conditioning, or quadrature costs are
independent of \(\varepsilon\).

The same trial space is compatible with both strong and variational
formulations. In a strong residual, the approximation must control
\(u\), \(\nabla u\), and \(D^2u\), and therefore uses the complete
derivative hierarchy in
\eqref{eq:intro_derivative_stable_rate}. In a variational formulation,
integration by parts reduces the required regularity to the transition and
its first derivative. The integrated Cauchy representation therefore
provides a common layer-adapted trial space for strong-form PINNs and
weak-form or hp-VPINN discretizations. The strong-form realization is
referred to as the Layer-Resolving XNet Physics-Informed Neural Network
(LRX-PINN). For variational experiments, the same representation is
embedded into existing stabilized hp-VPINN frameworks while their test
spaces, loss functionals, stabilization terms, quadrature rules, and
training settings are kept unchanged. This permits the effect of the trial
space to be separated from that of residual stabilization.

The numerical study is designed to test the theoretical mechanisms rather
than only final benchmark errors. It includes planar interior and boundary
layers, outflow and corner layers, curved interfaces, and a multiple-scale
problem containing both \(O(\varepsilon)\) exponential and
\(O(\sqrt{\varepsilon})\) characteristic layers. We compare integrated
Cauchy representations with standard tanh, arctangent, original Cauchy,
Fourier-feature, and KAN-based alternatives under matched training
conditions. In addition to global solution errors, we examine layer-region
errors, gradient accuracy, scaled errors as the physical layer thickness
decreases, and error versus the number of trainable parameters. The
variational experiments further determine whether trial-space adaptation
and SUPG-type stabilization provide complementary improvements.

The main contributions of this work are summarized as follows.

\begin{itemize}

\item
\textbf{A geometric integrated Cauchy layer trial space.}
We construct ridge atoms whose parameters explicitly encode the normal
direction, position, and physical width of a layer through
\[
    \boldsymbol n
    =
    \boldsymbol w/\|\boldsymbol w\|,
    \qquad
    s
    =
    -b/\|\boldsymbol w\|,
    \qquad
    \rho
    =
    d/\|\boldsymbol w\|.
\]
The same atom is transition-type at the solution level and generates
localized first- and second-derivative profiles on the physical scale
\(\rho\).

\item
\textbf{Derivative-stable exponential approximation.}
Using a Cauchy integral--trapezoidal quadrature argument, we prove that one
collection of integrated Cauchy atoms simultaneously approximates an
analytic stretched layer profile and its first two derivatives with a
common exponential rate.

\item
\textbf{Layer-thickness-uniform physical-space estimates.}
We establish scaled \(L^\infty\) and \(H^2\)-type approximation bounds
whose constants are independent of the layer thickness. The analysis
covers both \(O(\varepsilon)\) exponential layers and
\(O(\sqrt{\varepsilon})\) characteristic layers and yields
thickness-uniform best-approximation complexity for finite collections of
separated analytic layers.

\item
\textbf{Strong and variational realizations.}
We derive consequences for the strong convection--diffusion residual and
for the corresponding variational form. The same layer-adapted trial space
is implemented in a strong-form LRX-PINN and in stabilized hp-VPINN
frameworks, allowing trial-space effects to be assessed independently of
the residual formulation and stabilization strategy.

\item
\textbf{Numerical validation across layer types and scales.}
Experiments on interior, boundary, outflow, corner, curved, and
multiple-scale layers evaluate solution, gradient, layer-region, and
scaled errors. The results demonstrate that the integrated Cauchy trial
space improves accuracy and parameter efficiency across both strong and
variational formulations.

\end{itemize}

The remainder of the paper is organized as follows.
Section~\ref{sec:layer_trial_spaces} introduces the geometric integrated
Cauchy trial space and establishes the derivative-stable and
layer-thickness-uniform approximation results.
Section~\ref{sec:methodology} describes the strong and variational
realizations.
Section~\ref{sec:numerical_experiments} presents the numerical experiments
and comparisons.
Finally, Section~\ref{sec:conclusion} summarizes the conclusions and
discusses limitations and future directions.

\section{Layer-Thickness-Uniform Integrated Cauchy Approximation}
\label{sec:layer_trial_spaces}

Convection-dominated convection--diffusion problems develop thin and strongly
anisotropic solution structures. Across a layer, the solution undergoes an
$O(1)$ transition, whereas its first and second normal derivatives scale as
inverse powers of the physical layer thickness. Consequently, a trial space
that approximates only the solution values may still provide an inaccurate
representation of either the strong differential residual or the associated
variational form.

This section develops a layer-resolving trial space based on integrated
Cauchy ridge functions. The central point is not merely that the integrated
Cauchy function can be used as a neural activation. Rather, a single ridge
atom simultaneously represents a transition profile, its localized Cauchy
first derivative, and its localized sign-changing curvature. We show that,
for analytic layer profiles in the stretched coordinate, the same collection
of integrated Cauchy atoms approximates the profile and its first two
derivatives at one exponential rate. After rescaling to physical space, this
yields best-approximation estimates whose constants do not deteriorate as
the physical layer becomes thinner. The result applies to both
$O(\varepsilon)$ exponential layers and $O(\sqrt{\varepsilon})$
characteristic layers, with different atom widths matched to the respective
physical scales.

\subsection{Layer structures and integrated Cauchy atoms}
\label{subsec:canonical_layer_structures}

Consider the singularly perturbed convection--diffusion--reaction operator
\begin{equation}
    \mathcal L_\varepsilon u
    :=
    -\varepsilon \Delta u
    +
    \boldsymbol b(\boldsymbol x)\cdot\nabla u
    +
    c(\boldsymbol x)u,
    \qquad 0<\varepsilon\ll1.
    \label{eq:convection_diffusion_operator}
\end{equation}
Away from corners and other geometric singularities, the leading variation
of a thin layer can be described locally by a stretched normal coordinate.
We therefore begin with the canonical profile
\begin{equation}
    u_\delta(\boldsymbol x)
    =
    U(\tau),
    \qquad
    \tau
    :=
    \frac{\boldsymbol n\cdot\boldsymbol x-s}{\delta},
    \qquad
    \|\boldsymbol n\|=1,
    \label{eq:local_layer_profile}
\end{equation}
where $s$ specifies the layer location, $\boldsymbol n$ is its unit normal,
and $\delta>0$ is its physical thickness. The chain rule gives
\begin{equation}
    \nabla u_\delta
    =
    \delta^{-1}U'(\tau)\boldsymbol n,
    \qquad
    D^2u_\delta
    =
    \delta^{-2}U''(\tau)
    \boldsymbol n\otimes\boldsymbol n,
    \qquad
    \Delta u_\delta
    =
    \delta^{-2}U''(\tau).
    \label{eq:layer_gradient_hessian}
\end{equation}
Thus, for an order-one stretched profile,
\begin{equation}
    u_\delta=O(1),
    \qquad
    \nabla u_\delta=O(\delta^{-1}),
    \qquad
    D^2u_\delta=O(\delta^{-2}).
    \label{eq:layer_differential_scaling}
\end{equation}
Substitution into \eqref{eq:convection_diffusion_operator} yields
\begin{equation}
    \mathcal L_\varepsilon u_\delta
    =
    -\frac{\varepsilon}{\delta^2}U''(\tau)
    +
    \frac{\boldsymbol b(\boldsymbol x)\cdot\boldsymbol n}{\delta}
    U'(\tau)
    +
    c(\boldsymbol x)U(\tau).
    \label{eq:operator_on_layer_profile}
\end{equation}
Equation \eqref{eq:operator_on_layer_profile} identifies the three
structures that a layer-resolving trial space should reproduce: the bounded
transition $U$, the localized gradient $U'$, and the localized curvature
$U''$, all on the same physical scale $\delta$.

Two distinct layer scales are relevant for convection-dominated problems.
First, if the normal convection is nonzero, the balance
\begin{equation}
    \frac{\varepsilon}{\delta_E^2}
    \asymp
    \frac{|\boldsymbol b\cdot\boldsymbol n_E|}{\delta_E}
    \label{eq:exponential_layer_balance}
\end{equation}
gives the exponential-layer thickness
\begin{equation}
    \delta_E
    \asymp
    \frac{\varepsilon}{|\boldsymbol b\cdot\boldsymbol n_E|}.
    \label{eq:exponential_layer_width}
\end{equation}
In particular, when $|\boldsymbol b\cdot\boldsymbol n_E|=O(1)$,
$\delta_E=O(\varepsilon)$.

Second, on a characteristic boundary where
$\boldsymbol b\cdot\boldsymbol n_P=0$, the normal diffusion is balanced by
an order-one tangential transport or reaction scale. The resulting
characteristic, or parabolic, layer has thickness
\begin{equation}
    \delta_P\asymp\sqrt{\varepsilon}.
    \label{eq:characteristic_layer_width}
\end{equation}
The analysis below is formulated for a general thickness $\delta$ and is
then applied with $\delta=\delta_E$ and $\delta=\delta_P$. For a
characteristic layer with a slowly varying tangential amplitude, the result
applies to the fast normal profile uniformly in the tangential coordinate,
provided that the corresponding analyticity constants are uniform.

\paragraph{Integrated Cauchy layer atoms.}
\label{subsec:geometric_integrated_cauchy_atoms}

The one-dimensional Cauchy activation and its integrated form are
\begin{equation}
    \phi_{\lambda_1,\lambda_2,d}(z)
    =
    \frac{\lambda_1z+\lambda_2}{z^2+d^2},
    \qquad
    \Phi_{\lambda_1,\lambda_2,d}(z)
    =
    \frac{\lambda_1}{2}\log(z^2+d^2)
    +
    \frac{\lambda_2}{d}
    \arctan\!\left(\frac{z}{d}\right),
    \qquad d>0,
    \label{eq:cauchy_integrated_cauchy}
\end{equation}
with
\begin{equation}
    \Phi'_{\lambda_1,\lambda_2,d}
    =
    \phi_{\lambda_1,\lambda_2,d}.
    \label{eq:integrated_cauchy_derivative}
\end{equation}
Consider one ridge term with output coefficient $\alpha$,
\begin{equation}
    v(\boldsymbol x)
    =
    \alpha
    \Phi_{\lambda_1,\lambda_2,d}
    (\boldsymbol w^\top\boldsymbol x+b),
    \qquad \boldsymbol w\neq0.
    \label{eq:integrated_cauchy_ridge}
\end{equation}
Introduce
\begin{equation}
    \kappa:=\|\boldsymbol w\|,
    \qquad
    \boldsymbol n:=\frac{\boldsymbol w}{\|\boldsymbol w\|},
    \qquad
    s:=-\frac{b}{\|\boldsymbol w\|},
    \qquad
    \rho:=\frac{d}{\|\boldsymbol w\|}.
    \label{eq:geometric_parameters}
\end{equation}
Then
\begin{equation}
    \boldsymbol w^\top\boldsymbol x+b
    =
    \kappa(\boldsymbol n\cdot\boldsymbol x-s),
    \qquad
    d=\kappa\rho.
    \label{eq:affine_geometric_factorization}
\end{equation}
The parameters $\boldsymbol n$, $s$, and $\rho$ represent the layer normal,
location, and physical width, respectively.

\medskip
\noindent
\textbf{Proposition 1 (Geometric factorization).}
Up to an additive constant, the ridge term
\eqref{eq:integrated_cauchy_ridge} can be written as
\begin{equation}
    \Psi_{A,B,\boldsymbol n,s,\rho}(\boldsymbol x)
    =
    A\arctan\!\left(
        \frac{\boldsymbol n\cdot\boldsymbol x-s}{\rho}
    \right)
    +
    \frac{B}{2}
    \log\!\left(
        (\boldsymbol n\cdot\boldsymbol x-s)^2+\rho^2
    \right),
    \label{eq:geometric_layer_atom}
\end{equation}
where
\begin{equation}
    A=\frac{\alpha\lambda_2}{\kappa\rho},
    \qquad
    B=\alpha\lambda_1.
    \label{eq:geometric_amplitudes}
\end{equation}

\noindent
\textit{Proof.}
Using \eqref{eq:affine_geometric_factorization},
\begin{align}
    \alpha\Phi_{\lambda_1,\lambda_2,d}
    (\boldsymbol w^\top\boldsymbol x+b)
    & =
    \alpha\lambda_1\log\kappa
    +
    \frac{\alpha\lambda_1}{2}
    \log\!\left(
        (\boldsymbol n\cdot\boldsymbol x-s)^2+\rho^2
    \right)
    \notag\\
    &\quad+
    \frac{\alpha\lambda_2}{\kappa\rho}
    \arctan\!\left(
        \frac{\boldsymbol n\cdot\boldsymbol x-s}{\rho}
    \right).
    \label{eq:geometric_factorization_proof}
\end{align}
The first term is independent of $\boldsymbol x$ and can be absorbed into
the output bias. \hfill$\square$

The arctangent component carries a bounded transition across the hyperplane
\begin{equation}
    \Gamma_{\boldsymbol n,s}
    :=
    \{\boldsymbol x:\boldsymbol n\cdot\boldsymbol x=s\},
    \label{eq:layer_hyperplane}
\end{equation}
whereas the logarithmic component supplies an additional even rational
mode. In particular, the pure transition atom corresponds to $B=0$.
Let
\begin{equation}
    r(\boldsymbol x)
    :=
    \boldsymbol n\cdot\boldsymbol x-s,
    \qquad
    \tau:=\frac{r(\boldsymbol x)}{\rho}.
    \label{eq:atom_stretched_coordinate}
\end{equation}
Direct differentiation gives
\begin{equation}
    \nabla\Psi_{A,B,\boldsymbol n,s,\rho}
    =
    \rho^{-1}Q_1^{A,B}(\tau)\boldsymbol n,
    \qquad
    Q_1^{A,B}(\tau)
    :=
    \frac{A+B\tau}{1+\tau^2},
    \label{eq:geometric_atom_gradient}
\end{equation}
and
\begin{equation}
    D^2\Psi_{A,B,\boldsymbol n,s,\rho}
    =
    \rho^{-2}Q_2^{A,B}(\tau)
    \boldsymbol n\otimes\boldsymbol n,
    \qquad
    Q_2^{A,B}(\tau)
    :=
    \frac{-2A\tau+B(1-\tau^2)}{(1+\tau^2)^2}.
    \label{eq:geometric_atom_hessian}
\end{equation}
Consequently,
\begin{equation}
    \Delta\Psi_{A,B,\boldsymbol n,s,\rho}
    =
    \rho^{-2}Q_2^{A,B}(\tau).
    \label{eq:geometric_atom_laplacian}
\end{equation}
Thus one atom reproduces the complete differential hierarchy of a thin
layer: its value varies on the scale $\rho$, its gradient is a localized
Cauchy-type profile of magnitude $O(\rho^{-1})$, and its curvature is a
localized Cauchy-derivative profile of magnitude $O(\rho^{-2})$.

For the pure transition component,
\begin{equation}
    \frac{\partial}{\partial r}
    \left[
        A\arctan\!\left(\frac{r}{\rho}\right)
    \right]
    =
    \frac{A\rho}{r^2+\rho^2}.
    \label{eq:pure_transition_cauchy_kernel}
\end{equation}
Its derivative reaches half of its maximum magnitude at $|r|=\rho$, so the
full width at half maximum equals $2\rho$. Hence
\begin{equation}
    \boxed{\rho=\frac{d}{\|\boldsymbol w\|}}
    \label{eq:physical_neural_width}
\end{equation}
is the physical layer width represented by the neuron, rather than the
internal activation scale $d$ alone.

The corresponding geometric trial class is
\begin{equation}
    \mathcal V_M^{\mathrm{IC}}
    :=
    \left\{
        a_0+
        \sum_{m=1}^{M}
        \Psi_{A_m,B_m,\boldsymbol n_m,s_m,\rho_m}
        :
        \|\boldsymbol n_m\|=1,
        \ \rho_m>0
    \right\}.
    \label{eq:geometric_trial_class}
\end{equation}
Applying the differential operator to one atom gives
\begin{equation}
\begin{aligned}
    \mathcal L_\varepsilon
    \Psi_{A,B,\boldsymbol n,s,\rho}
    ={}&
    -\frac{\varepsilon}{\rho^2}Q_2^{A,B}(\tau)
    +
    \frac{\boldsymbol b(\boldsymbol x)\cdot\boldsymbol n}{\rho}
    Q_1^{A,B}(\tau)
    \\
    &+
    c(\boldsymbol x)
    \Psi_{A,B,\boldsymbol n,s,\rho}(\boldsymbol x).
\end{aligned}
    \label{eq:operator_on_geometric_atom}
\end{equation}
Comparison with \eqref{eq:operator_on_layer_profile} shows the structural
correspondence
\begin{equation}
    U\longleftrightarrow\Psi,
    \qquad
    U'\longleftrightarrow Q_1,
    \qquad
    U''\longleftrightarrow Q_2.
    \label{eq:differential_hierarchy_correspondence}
\end{equation}
To match a physical layer of thickness $\delta$, the atom widths should
therefore satisfy
\begin{equation}
    \rho_m\asymp\delta.
    \label{eq:atom_physical_width_matching}
\end{equation}

\subsection{Derivative-stable approximation in the stretched coordinate}
\label{subsec:complex_neighborhood_cauchy_approximation}

The simultaneous control of a layer profile and its derivatives requires a
Cauchy approximation that is uniform in a fixed complex neighborhood, not
only on the real interval. We therefore record the precise
Cauchy-integral--trapezoidal quadrature estimate used below.

\begin{lem}[Cauchy integral--trapezoidal quadrature]
\label{lem:cauchy_trapezoidal_quadrature}
Let $K\subset\mathbb C$ be compact, and let $\Gamma$ be a positively
oriented analytic Jordan curve whose interior $D$ satisfies
\begin{equation}
    K\Subset D.
    \label{eq:compact_inside_cauchy_contour}
\end{equation}
Assume that $\Gamma$ admits a $2\pi$-periodic parametrization
$\gamma:\mathbb R\to\Gamma$ that extends holomorphically to the strip
\begin{equation}
    S_a
    :=
    \{
        \theta\in\mathbb C:
        |\operatorname{Im}\theta|<a
    \}
    \label{eq:analytic_parameter_strip}
\end{equation}
for some $a>0$. Assume that there exists $a_1\in(0,a)$ such that $F$
is holomorphic in an open set containing
$\overline D\cup\gamma(\overline S_{a_1})$, where
$\overline S_{a_1}=\{\theta:|\operatorname{Im}\theta|\le a_1\}$.

Then there exists $a_0\in(0,a_1)$ such that, for every integer $N\ge1$, the
$N$-point trapezoidal Cauchy sum
\begin{equation}
    Q_N(z)
    :=
    \frac{1}{iN}
    \sum_{j=0}^{N-1}
    \frac{
        F(\zeta_j)\gamma'(\theta_j)
    }{
        \zeta_j-z
    },
    \qquad
    \theta_j
    =
    \frac{2\pi(j+\vartheta)}{N},
    \qquad
    \zeta_j=\gamma(\theta_j),
    \label{eq:trapezoidal_cauchy_sum}
\end{equation}
where $\vartheta\in[0,1)$ is any fixed phase shift, satisfies
\begin{equation}
    \sup_{z\in K}
    |F(z)-Q_N(z)|
    \le
    \frac{2H_{a_0}}{e^{a_0N}-1}
    \le
    C_{a_0}e^{-a_0N}.
    \label{eq:uniform_trapezoidal_cauchy_error}
\end{equation}
Here
\begin{equation}
    H_{a_0}
    :=
    \sup_{\substack{z\in K\\
                     |\operatorname{Im}\theta|\le a_0}}
    \left|
        \frac{
            F(\gamma(\theta))\gamma'(\theta)
        }{
            \gamma(\theta)-z
        }
    \right|
    <\infty,
    \label{eq:uniform_cauchy_integrand_bound}
\end{equation}
and the constants are independent of $N$.

Suppose in addition that $K$ and $\Gamma$ are symmetric with respect to the
real axis,
\begin{equation}
    F(\overline z)=\overline{F(z)},
    \qquad
    \gamma(-\theta)=\overline{\gamma(\theta)},
    \label{eq:conjugation_symmetry_assumptions}
\end{equation}
and that the parametrization is chosen so that
$\gamma(\theta)\in\mathbb R$ only for $\theta\in\pi\mathbb Z$.
For an even number $N=2M$, take the half-shifted nodes
$\vartheta=1/2$. They occur in conjugate pairs and none lies on the real
axis. Then $Q_{2M}$ can be written as an $M$-term
real Cauchy sum
\begin{equation}
    Q_{2M}(z)
    =
    \sum_{m=1}^{M}
    \frac{
        \lambda_{1,m}(z-\beta_m)+\lambda_{2,m}
    }{
        (z-\beta_m)^2+\widehat d_m^{\,2}
    },
    \qquad
    \beta_m\in\mathbb R,
    \quad
    \widehat d_m>0,
    \label{eq:paired_real_cauchy_sum}
\end{equation}
and
\begin{equation}
    \sup_{z\in K}
    |F(z)-Q_{2M}(z)|
    \le
    C\sigma^{-M},
    \qquad
    \sigma=e^{2a_0}>1.
    \label{eq:paired_cauchy_exponential_rate}
\end{equation}
\end{lem}

\begin{proof}
Since $K\Subset D$ and $\gamma$ is continuous on the real axis, the
separation
\begin{equation}
    \operatorname{dist}(K,\Gamma)>0
    \label{eq:real_contour_separation}
\end{equation}
holds. By continuity of the holomorphic extension of $\gamma$, one can
choose $a_0\in(0,a_1)$ sufficiently small such that
\begin{equation}
    d_0
    :=
    \inf_{\substack{z\in K\\
                     |\operatorname{Im}\theta|\le a_0}}
    |\gamma(\theta)-z|
    >0.
    \label{eq:complex_strip_contour_separation}
\end{equation}
For each fixed $z\in K$, define the $2\pi$-periodic function
\begin{equation}
    h_z(\theta)
    :=
    \frac{
        F(\gamma(\theta))\gamma'(\theta)
    }{
        \gamma(\theta)-z
    }.
    \label{eq:cauchy_periodic_integrand}
\end{equation}
The assumptions and \eqref{eq:complex_strip_contour_separation} imply that
$h_z$ is holomorphic in $S_{a_0}$ and satisfies
\begin{equation}
    \sup_{z\in K}
    \sup_{|\operatorname{Im}\theta|\le a_0}
    |h_z(\theta)|
    \le
    H_{a_0}.
    \label{eq:periodic_integrand_uniform_bound}
\end{equation}
By the Cauchy integral formula,
\begin{equation}
    F(z)
    =
    \frac{1}{2\pi i}
    \int_0^{2\pi}
    h_z(\theta)\,d\theta,
    \qquad z\in K.
    \label{eq:parametrized_cauchy_formula}
\end{equation}
The standard error bound for the $N$-point trapezoidal rule applied to a
$2\pi$-periodic function analytic and bounded by $H_{a_0}$ in
$S_{a_0}$ gives
\begin{equation}
    \left|
        \int_0^{2\pi}h_z(\theta)\,d\theta
        -
        \frac{2\pi}{N}
        \sum_{j=0}^{N-1}h_z(\theta_j)
    \right|
    \le
    \frac{4\pi H_{a_0}}{e^{a_0N}-1}.
    \label{eq:periodic_trapezoidal_error_bound}
\end{equation}
Dividing by $2\pi$ and using
\eqref{eq:parametrized_cauchy_formula} proves the first inequality in
\eqref{eq:uniform_trapezoidal_cauchy_error}. The second follows by
absorbing the finitely many small values of $N$ into the constant
$C_{a_0}$.

It remains to verify the real Cauchy representation. Under
\eqref{eq:conjugation_symmetry_assumptions}, conjugate quadrature nodes have
conjugate residues. More precisely, the coefficient of the pole
$\zeta_j$ in \eqref{eq:trapezoidal_cauchy_sum} is
\begin{equation}
    c_j
    :=
    \frac{1}{iN}
    F(\zeta_j)\gamma'(\theta_j).
    \label{eq:quadrature_pole_coefficient}
\end{equation}
For one conjugate pair, write
\begin{equation}
    \zeta_m
    =
    \beta_m+i\widehat d_m,
    \qquad
    c_m=p_m+iq_m,
    \qquad
    \widehat d_m>0.
    \label{eq:paired_pole_residue_notation}
\end{equation}
Then
\begin{equation}
\begin{aligned}
    \frac{c_m}{\zeta_m-z}
    +
    \frac{\overline{c_m}}{\overline{\zeta_m}-z}
    & =
    \frac{
        -2p_m(z-\beta_m)
        +2q_m\widehat d_m
    }{
        (z-\beta_m)^2+\widehat d_m^{\,2}
    }.
\end{aligned}
    \label{eq:conjugate_pair_to_real_cauchy_atom}
\end{equation}
Thus each conjugate pair is exactly one real Cauchy atom with
\begin{equation}
    \lambda_{1,m}=-2p_m,
    \qquad
    \lambda_{2,m}=2q_m\widehat d_m.
    \label{eq:real_cauchy_coefficients_from_residue}
\end{equation}
Pairing all $2M$ terms yields \eqref{eq:paired_real_cauchy_sum}. Finally,
setting $N=2M$ in \eqref{eq:uniform_trapezoidal_cauchy_error} gives
\eqref{eq:paired_cauchy_exponential_rate}.
\end{proof}

\begin{remark}
The phase shift in \eqref{eq:trapezoidal_cauchy_sum} does not change the
exponential trapezoidal-rule estimate. The half-shift used in the symmetric
case avoids the real-axis intersections at $\theta=0$ and $\theta=\pi$,
so every paired scale $\widehat d_m$ is strictly positive.
\end{remark}

\begin{remark}
The complex-neighborhood estimate in
Lemma~\ref{lem:cauchy_trapezoidal_quadrature} is essential. A value-only
approximation estimate on the real interval does not, by itself, control
the derivative of the approximation error.
\end{remark}

\paragraph{From Cauchy quadrature to integrated Cauchy approximation.}
\label{subsec:derivative_stable_approximation}

We now prove that one integrated Cauchy representation controls the layer
profile and its first two derivatives with the same exponential rate.

\begin{thm}[Derivative-stable integrated Cauchy approximation]
\label{thm:derivative_stable_integrated_cauchy}
Let $J\subset\mathbb R$ be a compact interval and let $\Omega_0$ be a
bounded domain, symmetric with respect to the real axis, such that
$J\Subset\Omega_0$. Assume that there exists a symmetric analytic Jordan
curve $\Gamma$ with interior $D$ satisfying
$\overline{\Omega_0}\Subset D$, and that $\Gamma$ has an analytic
parametrization satisfying the assumptions of
Lemma~\ref{lem:cauchy_trapezoidal_quadrature}. Suppose further that $U$ is
holomorphic in the corresponding open neighborhood of
$\overline D\cup\gamma(\overline S_{a_1})$ and satisfies
$U(\overline z)=\overline{U(z)}$. Then
there exist constants $C>0$ and $\sigma>1$, independent of $M$, and an
$M$-term integrated Cauchy function
\begin{equation}
\begin{aligned}
    G_M(t)
    =
    U(t_0)
    +
    \sum_{m=1}^{M}
    \Big[
        &\Phi_{\lambda_{1,m},\lambda_{2,m},\widehat d_m}
        (t-\beta_m)
        \\
        &-
        \Phi_{\lambda_{1,m},\lambda_{2,m},\widehat d_m}
        (t_0-\beta_m)
    \Big],
\end{aligned}
    \label{eq:integrated_cauchy_profile_approximation}
\end{equation}
where $t_0\in J$, such that
\begin{equation}
    \sum_{k=0}^{2}
    \|U^{(k)}-G_M^{(k)}\|_{L^\infty(J)}
    \le
    C\sigma^{-M}.
    \label{eq:stretched_coordinate_c2_rate}
\end{equation}
Consequently,
\begin{equation}
    \|U-G_M\|_{H^2(J)}
    \le
    C_J\sigma^{-M},
    \label{eq:stretched_coordinate_h2_rate}
\end{equation}
where $C_J$ is independent of $M$.
\end{thm}

\begin{proof}
Apply Lemma~\ref{lem:cauchy_trapezoidal_quadrature} to the
holomorphic derivative profile $F=U'$. There exists an $M$-term real Cauchy sum $R_M$ of
the form \eqref{eq:paired_real_cauchy_sum} satisfying
\begin{equation}
    \|U'-R_M\|_{L^\infty(\Omega_0)}
    \le
    C_1\sigma^{-M}.
    \label{eq:derivative_complex_approximation}
\end{equation}
Define $G_M$ by \eqref{eq:integrated_cauchy_profile_approximation}. Then
\begin{equation}
    G_M'=R_M,
    \qquad
    G_M(t_0)=U(t_0).
    \label{eq:integrated_derivative_identity}
\end{equation}
Hence
\begin{equation}
    \|U'-G_M'\|_{L^\infty(J)}
    \le
    C_1\sigma^{-M}.
    \label{eq:first_derivative_profile_rate}
\end{equation}
Moreover,
\begin{equation}
    U(t)-G_M(t)
    =
    \int_{t_0}^{t}
    \bigl(U'(\zeta)-R_M(\zeta)\bigr)\,d\zeta,
    \label{eq:value_error_integral_identity}
\end{equation}
so that
\begin{equation}
    \|U-G_M\|_{L^\infty(J)}
    \le
    |J|C_1\sigma^{-M}.
    \label{eq:value_profile_rate}
\end{equation}
Finally, $U'-R_M$ is holomorphic in $\Omega_0$. Since
$J\Subset\Omega_0$, Cauchy's derivative estimate gives
\begin{equation}
\begin{aligned}
    \|U''-G_M''\|_{L^\infty(J)}
    & =
    \|(U'-R_M)'\|_{L^\infty(J)}
    \\
    &\le
    C_{J,\Omega_0}
    \|U'-R_M\|_{L^\infty(\Omega_0)}
    \le
    C_2\sigma^{-M}.
\end{aligned}
    \label{eq:second_derivative_profile_rate}
\end{equation}
Combining \eqref{eq:first_derivative_profile_rate},
\eqref{eq:value_profile_rate}, and
\eqref{eq:second_derivative_profile_rate} proves
\eqref{eq:stretched_coordinate_c2_rate}. The $H^2(J)$ estimate follows
because $J$ is bounded.
\end{proof}

Theorem~\ref{thm:derivative_stable_integrated_cauchy} is the central
one-dimensional approximation result. The Cauchy functions approximate the
localized derivative profile $U'$, their antiderivatives approximate the
transition profile $U$, and holomorphic stability transfers the same
exponential rate to the curvature profile $U''$. Importantly, these are not
three independently fitted approximations: $G_M$, $G_M'$, and $G_M''$ are
generated by the same atom locations, scales, and coefficients.

\subsection{Thickness-uniform approximation and PDE consequences}
\label{subsec:layer_scaled_physical_approximation}

Let
\begin{equation}
    \Gamma
    \subset
    \{\boldsymbol x:\boldsymbol n\cdot\boldsymbol x=s\}
    \label{eq:layer_cross_section}
\end{equation}
be a bounded $(d-1)$-dimensional cross-section, and define the layer patch
\begin{equation}
    D_\delta
    :=
    \left\{
        \boldsymbol y+\delta t\boldsymbol n:
        \boldsymbol y\in\Gamma,
        \ t\in J
    \right\}.
    \label{eq:physical_layer_patch}
\end{equation}
Consider
\begin{equation}
    u_\delta(\boldsymbol x)
    =
    U\!\left(
        \frac{\boldsymbol n\cdot\boldsymbol x-s}{\delta}
    \right),
    \qquad \boldsymbol x\in D_\delta,
    \label{eq:physical_layer_target}
\end{equation}
and define
\begin{equation}
    g_{M,\delta}(\boldsymbol x)
    =
    G_M\!\left(
        \frac{\boldsymbol n\cdot\boldsymbol x-s}{\delta}
    \right).
    \label{eq:physical_integrated_cauchy_approximation}
\end{equation}

\begin{thm}[Uniform value--gradient--curvature approximation]
\label{thm:uniform_value_gradient_curvature}
Under the assumptions of
Theorem~\ref{thm:derivative_stable_integrated_cauchy},
\begin{equation}
\begin{aligned}
    &\|u_\delta-g_{M,\delta}\|_{L^\infty(D_\delta)}
    +
    \delta
    \|\nabla u_\delta-\nabla g_{M,\delta}\|_{L^\infty(D_\delta)}
    \\
    &\qquad+
    \delta^2
    \|D^2u_\delta-D^2g_{M,\delta}\|_{L^\infty(D_\delta)}
    \le
    C\sigma^{-M},
\end{aligned}
    \label{eq:uniform_c2_physical_rate}
\end{equation}
where $C$ and $\sigma$ are independent of $M$ and $\delta$.
\end{thm}

\begin{proof}
Let
\begin{equation}
    t
    =
    \frac{\boldsymbol n\cdot\boldsymbol x-s}{\delta},
    \qquad
    E_M(t):=U(t)-G_M(t).
    \label{eq:physical_stretched_error}
\end{equation}
Then
\begin{equation}
    u_\delta-g_{M,\delta}=E_M(t),
    \qquad
    \nabla(u_\delta-g_{M,\delta})
    =
    \delta^{-1}E_M'(t)\boldsymbol n,
    \label{eq:physical_error_gradient}
\end{equation}
and
\begin{equation}
    D^2(u_\delta-g_{M,\delta})
    =
    \delta^{-2}E_M''(t)
    \boldsymbol n\otimes\boldsymbol n.
    \label{eq:physical_error_hessian}
\end{equation}
The result follows directly from
\eqref{eq:stretched_coordinate_c2_rate}.
\end{proof}

For an $L^2$ formulation, introduce the normal layer-scaled norm
\begin{equation}
\begin{aligned}
    \|v\|_{\mathcal X_\delta(D_\delta)}^2
    :=
    &\,
    \delta^{-1}\|v\|_{L^2(D_\delta)}^2
    +
    \delta
    \|\partial_{\boldsymbol n}v\|_{L^2(D_\delta)}^2
    \\
    &+
    \delta^3
    \|\partial_{\boldsymbol n\boldsymbol n}v\|_{L^2(D_\delta)}^2.
\end{aligned}
    \label{eq:layer_scaled_norm}
\end{equation}

\begin{thm}[Layer-scaled exponential approximation]
\label{thm:layer_scaled_exponential_approximation}
Under the assumptions of
Theorem~\ref{thm:derivative_stable_integrated_cauchy},
\begin{equation}
    \|u_\delta-g_{M,\delta}\|_{\mathcal X_\delta(D_\delta)}
    \le
    C_\Gamma\sigma^{-M},
    \label{eq:layer_scaled_exponential_rate}
\end{equation}
where $C_\Gamma$ and $\sigma$ are independent of both $M$ and $\delta$.
If the $m$th stretched-coordinate atom has center $\beta_m$ and scale
$\widehat d_m$, then the corresponding physical atom has center and width
\begin{equation}
    s_m=s+\delta\beta_m,
    \qquad
    \rho_m=\delta\widehat d_m.
    \label{eq:stretched_to_physical_atom_geometry}
\end{equation}
\end{thm}

\begin{proof}
Under the parametrization
\begin{equation}
    \boldsymbol x
    =
    \boldsymbol y+\delta t\boldsymbol n,
    \qquad
    \boldsymbol y\in\Gamma,
    \quad t\in J,
    \label{eq:layer_patch_parametrization}
\end{equation}
the volume element is
$d\boldsymbol x=\delta\,d\sigma(\boldsymbol y)\,dt$. Therefore,
\begin{equation}
\begin{aligned}
    \delta^{-1}
    \|u_\delta-g_{M,\delta}\|_{L^2(D_\delta)}^2
    & =
    |\Gamma|\|E_M\|_{L^2(J)}^2,
    \\
    \delta
    \|\partial_{\boldsymbol n}(u_\delta-g_{M,\delta})\|_{L^2(D_\delta)}^2
    & =
    |\Gamma|\|E_M'\|_{L^2(J)}^2,
    \\
    \delta^3
    \|\partial_{\boldsymbol n\boldsymbol n}(u_\delta-g_{M,\delta})\|_{L^2(D_\delta)}^2
    & =
    |\Gamma|\|E_M''\|_{L^2(J)}^2.
\end{aligned}
    \label{eq:scaled_norm_identity_components}
\end{equation}
Hence
\begin{equation}
    \|u_\delta-g_{M,\delta}\|_{\mathcal X_\delta(D_\delta)}
    =
    |\Gamma|^{1/2}
    \|U-G_M\|_{H^2(J)},
    \label{eq:scaled_norm_exact_identity}
\end{equation}
and \eqref{eq:layer_scaled_exponential_rate} follows from
\eqref{eq:stretched_coordinate_h2_rate}. Finally,
\begin{equation}
    t-\beta_m
    =
    \frac{\boldsymbol n\cdot\boldsymbol x-(s+\delta\beta_m)}{\delta},
    \label{eq:physical_center_mapping}
\end{equation}
so the physical center and width are given by
\eqref{eq:stretched_to_physical_atom_geometry}.
\end{proof}

For a prescribed tolerance $\eta>0$, it is sufficient to choose
\begin{equation}
    M
    \ge
    \frac{\log(C_\Gamma/\eta)}{\log\sigma}.
    \label{eq:layer_complexity_bound}
\end{equation}
Thus the number of layer atoms is logarithmic in the target accuracy and,
at the level of best approximation, independent of the inverse physical
layer thickness $\delta^{-1}$.

\paragraph{Exponential and characteristic layers.}
\label{subsec:exponential_characteristic_layers}

The preceding result applies to both principal layer scales of a
convection-dominated problem.

\medskip
\noindent
\textbf{Corollary 1 (Exponential layers).}
Let
\begin{equation}
    u_\varepsilon^E(\boldsymbol x)
    =
    U_E\!\left(
        \frac{\boldsymbol n_E\cdot\boldsymbol x-s_E}{\delta_E}
    \right),
    \qquad
    \delta_E
    \asymp
    \frac{\varepsilon}{|\boldsymbol b\cdot\boldsymbol n_E|},
    \label{eq:exponential_layer_profile}
\end{equation}
where $U_E$ satisfies the assumptions of
Theorem~\ref{thm:derivative_stable_integrated_cauchy}. Then there exists an
$M$-term integrated Cauchy approximation $g_{M,\varepsilon}^E$ such that
\begin{equation}
    \|u_\varepsilon^E-g_{M,\varepsilon}^E\|_{\mathcal X_{\delta_E}}
    \le
    C_E\sigma_E^{-M},
    \label{eq:exponential_layer_uniform_rate}
\end{equation}
with constants independent of $\varepsilon$. The effective widths satisfy
\begin{equation}
    \rho_m^E=\delta_E\widehat d_m=O(\varepsilon)
    \label{eq:exponential_atom_widths}
\end{equation}
whenever $|\boldsymbol b\cdot\boldsymbol n_E|=O(1)$.

\medskip
\noindent
\textbf{Corollary 2 (Characteristic layers).}
Let
\begin{equation}
    u_\varepsilon^P(\boldsymbol x)
    =
    U_P\!\left(
        \frac{\boldsymbol n_P\cdot\boldsymbol x-s_P}{\sqrt\varepsilon}
    \right),
    \qquad
    \boldsymbol b\cdot\boldsymbol n_P=0,
    \label{eq:characteristic_layer_profile}
\end{equation}
where $U_P$ satisfies the assumptions of
Theorem~\ref{thm:derivative_stable_integrated_cauchy}. Then there exists an
$M$-term integrated Cauchy approximation $g_{M,\varepsilon}^P$ such that
\begin{equation}
    \|u_\varepsilon^P-g_{M,\varepsilon}^P\|_{\mathcal X_{\sqrt\varepsilon}}
    \le
    C_P\sigma_P^{-M},
    \label{eq:characteristic_layer_uniform_rate}
\end{equation}
with constants independent of $\varepsilon$. The effective widths satisfy
\begin{equation}
    \rho_m^P=\sqrt\varepsilon\,\widehat d_m
    =O(\sqrt\varepsilon).
    \label{eq:characteristic_atom_widths}
\end{equation}

The two corollaries show that one integrated Cauchy family accommodates two
different physical scales through different learned or prescribed values of
$\rho$. The method does not impose the incorrect universal rule
$\rho=O(\varepsilon)$: exponential layers require
$\rho=O(\varepsilon)$, whereas characteristic layers require
$\rho=O(\sqrt\varepsilon)$.

\paragraph{Coexisting layers and representation complexity.}
\label{subsec:coexisting_layers_complexity}

Consider a solution decomposition on mutually separated local patches,
\begin{equation}
\begin{aligned}
    u_\varepsilon
    =
    u_\varepsilon^{\mathrm{reg}}
    &+
    \sum_{\ell=1}^{N_E}
    U_{E,\ell}\!\left(
        \frac{\boldsymbol n_{E,\ell}\cdot\boldsymbol x-s_{E,\ell}}
             {\delta_{E,\ell}}
    \right)
    \\
    &+
    \sum_{q=1}^{N_P}
    U_{P,q}\!\left(
        \frac{\boldsymbol n_{P,q}\cdot\boldsymbol x-s_{P,q}}
             {\delta_{P,q}}
    \right),
\end{aligned}
    \label{eq:coexisting_layer_decomposition}
\end{equation}
where
\begin{equation}
    \delta_{E,\ell}=O(\varepsilon),
    \qquad
    \delta_{P,q}=O(\sqrt\varepsilon),
    \label{eq:coexisting_layer_scales}
\end{equation}
and $N_E$ and $N_P$ are independent of $\varepsilon$. Assume that all
stretched profiles satisfy the analyticity assumptions above with uniform
constants. Approximating each layer component separately and adding the
results gives the following consequence.

\begin{thm}[Thickness-independent layer complexity]
\label{thm:thickness_independent_layer_complexity}
For every tolerance $\eta>0$, the layer part of
\eqref{eq:coexisting_layer_decomposition} admits an integrated Cauchy
approximation with total number of atoms
\begin{equation}
    M_{\mathrm{layer}}
    \le
    C_0(N_E+N_P)
    \log\!\left(\frac{C_1}{\eta}\right),
    \label{eq:total_layer_complexity}
\end{equation}
such that the sum of the corresponding local layer-scaled errors is at most
$\eta$. The constants $C_0$ and $C_1$ are independent of
$\min_\ell\delta_{E,\ell}$,
$\min_q\delta_{P,q}$, and hence independent of the physical layer
thicknesses.
\end{thm}

\begin{proof}
Apply Theorem~\ref{thm:layer_scaled_exponential_approximation} to every
layer profile, with tolerance $\eta/(N_E+N_P)$. Each component requires at
most $C\log(C'(N_E+N_P)/\eta)$ atoms. Summing over the finite number of
layers yields \eqref{eq:total_layer_complexity}.
\end{proof}

Theorem~\ref{thm:thickness_independent_layer_complexity} is a
best-approximation statement. It does not claim that optimization
iterations, conditioning, or quadrature cost are independent of
$\varepsilon$. Those issues concern the realization of the trial space and
must be assessed separately.

\paragraph{Consequences for strong and variational formulations.}
\label{subsec:strong_variational_consequences}

We finally connect the approximation result to the two residual
formulations used in the numerical method.

\medskip\noindent\textit{Strong residual.}

Let
\begin{equation}
    e_{M,\delta}
    :=
    u_\delta-g_{M,\delta}.
    \label{eq:layer_approximation_error}
\end{equation}
Since $e_{M,\delta}=E_M(t)$ on the flat layer patch,
\begin{equation}
    \nabla e_{M,\delta}
    =
    \delta^{-1}E_M'(t)\boldsymbol n,
    \qquad
    \Delta e_{M,\delta}
    =
    \delta^{-2}E_M''(t).
    \label{eq:layer_error_derivatives}
\end{equation}

\begin{thm}[Layer-scaled strong-residual estimate]
\label{thm:layer_scaled_strong_residual}
Assume that $\boldsymbol b,c\in L^\infty(D_\delta)$. Then
\begin{equation}
\begin{aligned}
    \delta^{1/2}
    \|\mathcal L_\varepsilon e_{M,\delta}\|_{L^2(D_\delta)}
    \le
    |\Gamma|^{1/2}
    \Big(
        &\frac{\varepsilon}{\delta}
        \|E_M''\|_{L^2(J)}
        \\
        &+
        \|\boldsymbol b\cdot\boldsymbol n\|_{L^\infty(D_\delta)}
        \|E_M'\|_{L^2(J)}
        \\
        &+
        \delta
        \|c\|_{L^\infty(D_\delta)}
        \|E_M\|_{L^2(J)}
    \Big).
\end{aligned}
    \label{eq:normalized_layer_residual_estimate}
\end{equation}
Consequently,
\begin{equation}
\begin{aligned}
    \delta^{1/2}
    \|\mathcal L_\varepsilon
    (u_\delta-g_{M,\delta})\|_{L^2(D_\delta)}
    \le
    C
    \left(
        \frac{\varepsilon}{\delta}
        +
        \|\boldsymbol b\cdot\boldsymbol n\|_{L^\infty(D_\delta)}
        +
        \delta\|c\|_{L^\infty(D_\delta)}
    \right)
    \sigma^{-M}.
\end{aligned}
    \label{eq:operator_scaled_exponential_rate}
\end{equation}
\end{thm}

\begin{proof}
Using \eqref{eq:layer_error_derivatives},
\begin{equation}
    \mathcal L_\varepsilon e_{M,\delta}
    =
    -\frac{\varepsilon}{\delta^2}E_M''(t)
    +
    \frac{\boldsymbol b(\boldsymbol x)\cdot\boldsymbol n}{\delta}
    E_M'(t)
    +
    c(\boldsymbol x)E_M(t).
    \label{eq:operator_on_layer_error}
\end{equation}
The change of variables
$d\boldsymbol x=\delta\,d\sigma(\boldsymbol y)\,dt$ and the triangle
inequality give \eqref{eq:normalized_layer_residual_estimate}. The second
estimate follows from
Theorem~\ref{thm:derivative_stable_integrated_cauchy}.
\end{proof}

For an exponential layer,
$\delta=\delta_E\asymp\varepsilon/|\boldsymbol b\cdot\boldsymbol n_E|$,
so the factor in parentheses in
\eqref{eq:operator_scaled_exponential_rate} remains bounded under the
usual order-one coefficient assumptions. For a canonical characteristic
layer with $\delta=\sqrt\varepsilon$ and
$\boldsymbol b\cdot\boldsymbol n_P=0$, the same factor is
$O(\sqrt\varepsilon)$ up to the bounded reaction coefficient. Thus both
layer classes inherit an exponential strong-residual rate from the same
one-dimensional derivative-stable approximation theorem.

\medskip\noindent\textit{Variational residual.}

Let
\begin{equation}
    a_D(w,v)
    :=
    \int_{D_\delta}
    \left(
        \varepsilon\nabla w\cdot\nabla v
        +
        (\boldsymbol b\cdot\nabla w)v
        +
        cwv
    \right)d\boldsymbol x
    \label{eq:local_variational_form}
\end{equation}
be the local bilinear form on the layer patch. The weak formulation only
requires the value and first derivative of the trial function. Define the
first-order layer-scaled norm
\begin{equation}
    \|v\|_{\mathcal Y_\delta(D_\delta)}^2
    :=
    \delta^{-1}\|v\|_{L^2(D_\delta)}^2
    +
    \delta
    \|\partial_{\boldsymbol n}v\|_{L^2(D_\delta)}^2.
    \label{eq:first_order_layer_scaled_norm}
\end{equation}
Theorem~\ref{thm:layer_scaled_exponential_approximation} immediately gives
\begin{equation}
    \|u_\delta-g_{M,\delta}\|_{\mathcal Y_\delta(D_\delta)}
    \le
    C\sigma^{-M}.
    \label{eq:first_order_layer_scaled_rate}
\end{equation}
Moreover, for every $v\in H^1(D_\delta)$,
\begin{equation}
\begin{aligned}
    |a_D(e_{M,\delta},v)|
    \le{}&
    \varepsilon
    \|\nabla e_{M,\delta}\|_{L^2(D_\delta)}
    \|\nabla v\|_{L^2(D_\delta)}
    \\
    &+
    \|\boldsymbol b\|_{L^\infty(D_\delta)}
    \|\nabla e_{M,\delta}\|_{L^2(D_\delta)}
    \|v\|_{L^2(D_\delta)}
    \\
    &+
    \|c\|_{L^\infty(D_\delta)}
    \|e_{M,\delta}\|_{L^2(D_\delta)}
    \|v\|_{L^2(D_\delta)}.
\end{aligned}
    \label{eq:variational_residual_continuity}
\end{equation}
Using the scaled estimates for $e_{M,\delta}$ and its gradient yields
\begin{equation}
\begin{aligned}
    \delta^{1/2}
    \sup_{0\neq v\in H^1(D_\delta)}
    \frac{|a_D(e_{M,\delta},v)|}
         {\|v\|_{H^1(D_\delta)}}
    \le
    C
    \left(
        \varepsilon
        +
        \|\boldsymbol b\|_{L^\infty(D_\delta)}
        +
        \delta\|c\|_{L^\infty(D_\delta)}
    \right)
    \sigma^{-M}.
\end{aligned}
    \label{eq:variational_residual_exponential_rate}
\end{equation}
Hence the same integrated Cauchy trial space is compatible with both
formulations:
\begin{equation}
    \boxed{
    \begin{aligned}
        \text{strong form:}\quad
        &G_M,\ G_M',\ G_M'',
        \\
        \text{variational form:}\quad
        &G_M,\ G_M'.
    \end{aligned}}
    \label{eq:strong_weak_compatibility}
\end{equation}
The strong formulation exploits the full value--gradient--curvature
hierarchy, whereas integration by parts reduces the weak formulation to the
transition and localized-gradient levels.

\begin{remark}
The estimates in this section establish approximation and consistency of
the integrated Cauchy trial space. They do not, by themselves, imply a
discrete inf--sup condition, quasi-optimality of a particular VPINN
realization, or convergence of a nonconvex optimizer. Those solver-level
questions are separate from the layer-thickness-uniform best-approximation
property proved here.
\end{remark}

\begin{remark}[Curved layers]
Let $r(\boldsymbol x)$ be the signed-distance function to a smooth
interface and consider
\begin{equation}
    u_\delta(\boldsymbol x)
    =
    U\!\left(\frac{r(\boldsymbol x)}{\delta}\right).
    \label{eq:curved_layer_profile}
\end{equation}
In a tubular neighborhood where $\|\nabla r\|=1$,
\begin{equation}
    D^2u_\delta
    =
    \delta^{-2}U''\!\left(\frac r\delta\right)
    \nabla r\otimes\nabla r
    +
    \delta^{-1}U'\!\left(\frac r\delta\right)D^2r,
    \label{eq:curved_layer_hessian}
\end{equation}
and
\begin{equation}
    \Delta u_\delta
    =
    \delta^{-2}U''\!\left(\frac r\delta\right)
    +
    \delta^{-1}(\Delta r)
    U'\!\left(\frac r\delta\right).
    \label{eq:curved_layer_laplacian}
\end{equation}
The leading term retains the same normal $O(\delta^{-2})$ structure as the
planar analysis, while curvature contributes a lower-order
$O(\delta^{-1})$ correction. Consequently, affine integrated Cauchy atoms
provide a patchwise approximation of a smooth curved layer, with different
local normals used along the interface.
\end{remark}

Table~\ref{tab:layer_activation} summarizes how the integrated Cauchy
activation, its derivative, and its second derivative match the dominant
layer structures in the benchmark problems considered below.

\begin{table}[H]
\centering
\caption{Layer structures in Tasks 2--6 and their matching with the integrated Cauchy activation. The common mechanism is value-level transition together with derivative-level localization effects.}
\footnotesize
\resizebox{\textwidth}{!}{
\begin{tabular}{m{0.08\textwidth} m{0.16\textwidth} m{0.27\textwidth} m{0.24\textwidth} m{0.29\textwidth}}
\toprule
Task & Layer type & Representative layer component & Dominant derivative structure & Integrated Cauchy matching \\
\midrule
Task 2 & Planar interior layer & $T((x-1/2)/\varepsilon)$ with $T=\tanh$ & $u_x=O(\varepsilon^{-1})$ localized near $x=1/2$ & The affine coordinate aligns with the $x$-normal direction; $d/\|w\|=O(\varepsilon)$ resolves the layer width. \\[4pt]
\midrule
Task 3 & Boundary layer & $(1-\exp(-(1-x)/\varepsilon))\sin(\pi y)$ & $u_x=O(\varepsilon^{-1})$ localized near $x=1$ & The activation value represents the boundary transition, while the derivative represents the sharp normal layer derivative. \\[4pt]
\midrule
Task 4 & Mixed interior-boundary layer & $\tanh((x-1/2)/\varepsilon)+1-\exp(-(1-x)/\varepsilon)$ & Superposition of interior-layer and boundary-layer localized derivatives & Multiple neurons can align with different layer locations and scales, resolving coexisting localized structures. \\[4pt]
\midrule
Task 5 & Outflow layers and corner layer & $\exp(2(x-1)/\varepsilon)$, $\exp(3(y-1)/\varepsilon)$, and $\exp((2(x-1)+3(y-1))/\varepsilon)$ & Localized derivatives near outflow boundaries and the corner interaction region & Different affine directions align with $x$-outflow, $y$-outflow, and combined corner-layer directions. \\[4pt]
\midrule
Task 6 & Circular interior layer & $\arctan(200(0.25^2-(x-0.5)^2-(y-0.5)^2))$ & $\|\nabla u\|$ localized near a circular interface & The curved layer is represented patch by patch; each neuron learns a local normal coordinate through its affine projection. \\
\bottomrule
\end{tabular}}
\label{tab:layer_activation}
\end{table}

\section{Experimental Framework}
\subsection{Layer-Resolving XNet Physics-Informed Neural Network}

In this subsection, we briefly introduce the proposed LRX-PINN framework and its layer-adapted integrated Cauchy activation for convection-dominated convection--diffusion problems. We consider a general boundary value problem of the form
\begin{equation}
\mathcal{N}[u](\mathbf{x}) = f(\mathbf{x}), \quad \mathbf{x} \in \Omega,
\end{equation}
subject to boundary (or initial) conditions
\begin{equation*}
\mathcal{B}[u](\mathbf{x}) = g(\mathbf{x}), \quad \mathbf{x} \in \partial \Omega,
\end{equation*}
where \(\mathcal{N}\) denotes a differential operator and \(\mathcal{B}\) represents the associated boundary operator.

Within the PINN framework, the solution \(u(\mathbf{x})\) is approximated by a neural network \(u_\theta(\mathbf{x})\), whose parameters are determined by minimizing the residuals of the governing equation and boundary conditions.

In the proposed LRX-PINN, the network is constructed using integrated
Cauchy activation functions \(\Phi_d\). The network is trained by minimizing a loss function that enforces both the governing equation and the boundary conditions.

The corresponding PDE residual is defined as
\begin{equation*}
\mathcal{R}_\theta(\mathbf{x}) := \mathcal{N}[u_\theta](\mathbf{x}) - f(\mathbf{x}), \quad \mathbf{x} \in \Omega.
\end{equation*}

The total loss function is given by
\[
\mathcal{L}(\theta)
=
\lambda_f\mathcal{L}_{\text{PDE}}
+
\lambda_b\mathcal{L}_{\text{BC}},
\]
where
\[
\mathcal{L}_{\text{PDE}}
=
\frac{1}{N_f}
\sum_{i=1}^{N_f}
\left|
\mathcal{R}_\theta(\mathbf{x}_i^f)
\right|^2,
\]
and
\[
\mathcal{L}_{\text{BC}}
=
\frac{1}{N_b}
\sum_{i=1}^{N_b}
\left|
\mathcal{B}[u_\theta](\mathbf{x}_i^b)-g(\mathbf{x}_i^b)
\right|^2.
\]
Here, $\{\mathbf{x}_i^f\}_{i=1}^{N_f} \subset \Omega$ and $\{\mathbf{x}_i^b\}_{i=1}^{N_b} \subset \partial \Omega$ denote collocation points in the interior and on the boundary, respectively. The network parameters $\theta$ are optimized via gradient-based methods, where derivatives of $u_\theta$ are computed using automatic differentiation.

For convection-dominated convection--diffusion problems with $0<\varepsilon\ll1$,
the solution typically develops thin boundary or interior layers of width \(O(\varepsilon)\), leading to highly localized transition structures. Motivated by the layer-adapted analysis presented in Section 2, we employ the following integrated Cauchy activation:
\begin{equation}\label{modify}
\Phi_d(z)
=
\frac{\lambda_1}{2}\log(z^2+d^2)
+
\frac{\lambda_2}{d}
\arctan\!\left(\frac{z}{d}\right).
\end{equation}

For notational simplicity, we write \(\Phi_d(z)\) instead of
\(\Phi_{\lambda_1,\lambda_2,d}(z)\) when no confusion arises. The derivative of the activation satisfies $\Phi_d'(z)=\phi_d(z)$, 
thereby preserving the localized structure of the original Cauchy kernel while introducing a transition-type representation. Moreover, the scale parameter \(d\) controls the effective transition width of the activation and can be selected according to $\frac{d}{\|w\|} = O(\varepsilon)$, allowing the activation scale to adapt to the intrinsic width of the layer structure. Compared with the original Cauchy activation, the proposed variant provides
a more suitable neural representation for thin-layer solutions and leads to
improved accuracy and robustness in the small-diffusion regime.

\subsection{Benchmark Problems}
In this subsection, we present six representative convection--diffusion problems and classify them into three categories according to the complexity of their layer structures and the corresponding approximation difficulty.

\textbf{Regular problem.}\enspace
Task~1 is a smooth convection-diffusion problem without pronounced layer structures.

\textbf{Mild convection-dominated problems.}\enspace
Task~2 develops a sharp interior transition layer;
Task~3 forms thin boundary layers near the domain boundary;
Task~4 combines both, producing multiscale interior-boundary layer structures.
Together with the smooth case in Task~1, these benchmarks provide a progressive test set for evaluating the approximation capability of LRX-PINN under increasing layer complexity.

\textbf{Strongly convection-dominated problems.}\enspace
Task~5 \citep{john1997} exhibits thin outflow-boundary layers with highly localized
gradients;
Task~6 \citep{john1997} contains a circular interior layer with a curved
sharp-transition structure.
For these problems, the integrated Cauchy activation is embedded into more advanced
PINN frameworks.

Table~\ref{tab:layer_activation} clarifies the representation-oriented
route of this work. We begin from the singularly perturbed PDE structure,
identify the layer mechanisms that drive the approximation difficulty,
construct a Cauchy-based basis whose value and derivative match these
structures, and then test the resulting layer-adapted trial space across
planar, boundary, mixed, outflow, corner, and curved layers.

\subsection{Experimental Setup}

We compare the proposed LRX-PINN with two representative PINN-based approaches: PIKAN\footnote{Code available at \url{https://github.com/wanjiashan/PIKANs}}, based on Kolmogorov--Arnold representations, and Fourier PINN\footnote{Code available at \url{https://github.com/songc0a/Fourier-feature-PINN-based-multifrequency-multisource-Helmholtz-solver}}, based on Fourier feature embeddings. The implementations are adapted from the publicly available codes in \cite{wang2024kinn}. For the strongly convection-dominated benchmarks (Tasks 5--6), we further incorporate the proposed integrated Cauchy activation into hp-VPINN/hp-fastVPINN frameworks\footnote{Original fastVPINN code available at \url{https://github.com/airexlab/fastvpinns}} while keeping the other components unchanged.


For Tasks 1--4, all methods are trained using 8000 Adam iterations followed by 12000 L-BFGS iterations. As observed in \cite{si2026}, the L-BFGS stage is important for obtaining high-accuracy PINN solutions. We use fixed loss weights $\lambda_f=1$ and $\lambda_b=40$ for Tasks 1--4. No adaptive loss-balancing strategy is used, so that the effect of the proposed neural representation can be isolated more clearly.

For Tasks 5--6, following the settings in \cite{anandh2025,frerichs2026}, we employ the prescribed 50,000 or 100,000 Adam iterations and keep the original loss weights, stabilization parameters, and training schedules unchanged.

\paragraph{Computational Cost.}

Although LRX-PINN uses a much smaller number of trainable parameters than the baseline networks, its per-neuron computational cost is higher. For a single hidden layer with $N$ neurons and input dimension $d_{\rm in}$, the affine part has cost $O(Nd_{\rm in})$, while the activation evaluation remains linear in $N$ but has a larger constant factor because the integrated Cauchy activation contains rational, logarithmic, and arctangent operations. PINN training further requires automatic differentiation of first- and second-order derivatives, so the wall-clock time is not solely determined by parameter count. Therefore, LRX-PINN should be interpreted as an accuracy--parameter efficient model rather than a uniformly faster model. In the reported convection-dominated benchmarks, this additional constant cost is compensated by improved accuracy and a substantially smaller trainable representation.

\paragraph{Network Configurations.}

The proposed LRX-PINN adopts a single-hidden-layer architecture with $N_{\text{Cauchy}}=120$ neurons. Each neuron is equipped with trainable Cauchy activation parameters $\{\lambda_1,\lambda_2,d\}$, allowing the network to adaptively capture localized structures and sharp layers. For extremely thin-layer problems, the modified Cauchy activation $\Phi_d$ in (\ref{modify}) is employed. The resulting model contains $841$ trainable parameters.

PIKAN employs a residual KAN architecture with $D_{\text{KAN}}=3$ hidden layers of width $N_{\text{KAN}}=20$ and $N_{\text{basis}}=7$ Gaussian RBF basis functions per coordinate. Residual skip connections and layer normalization are incorporated to improve optimization stability. The total number of trainable parameters is $7065$.

Fourier PINN employs a Fourier feature embedding with mapping size $M_{\mathrm{Fourier}}=16$, followed by a fully connected network with $D_{\mathrm{Fourier}}=3$ hidden layers of width $N_{\mathrm{Fourier}}=32$. The Fourier scale parameter $\sigma$ is selected according to $\varepsilon$ and remains fixed during training. The resulting model contains 3201 trainable parameters.

\begin{table}[H]
\centering
\caption{Network configurations and parameter counts.}
\label{tab:model_size}
\begin{tabular}{lcc}
\toprule
Method & Architecture & Params \\
\midrule
LRX-PINN & 1 hidden layer, $N_{\text{Cauchy}}=120$ & 841 \\
PIKAN & $D_{\text{KAN}}=3$, $N_{\text{KAN}}=20$, $N_{\text{basis}}=7$ & 7065 \\
Fourier PINN & $D_{\text{Fourier}}=3$, $N_{\text{Fourier}}=32$, $M_{\text{Fourier}}=16$ & 3201 \\
\bottomrule
\end{tabular}\label{para}
\end{table}

The parameter reduction is substantial. Compared with PIKAN and Fourier PINN, the parameter ratios are
$$
\frac{N_{\rm PIKAN}}{N_{\rm LR}}=\frac{7065}{841}\approx 8.40,
\qquad
\frac{N_{\rm Fourier}}{N_{\rm LR}}=\frac{3201}{841}\approx 3.81.
$$
Equivalently, LRX-PINN uses only about $11.9\%$ of the PIKAN parameters and $26.3\%$ of the Fourier PINN parameters. Therefore, the comparison is not only an accuracy comparison, but also a parameter-efficiency comparison. Table~\ref{tab:model_size} summarizes the configurations and parameter counts of the compared models. We adopt a single-hidden-layer LRX-PINN with 120 neurons, which provides a compact architecture while retaining sufficient representational capacity for the benchmark problems considered in this work. To further highlight the advantages of the proposed method, we use baseline architectures with moderate depth and increased width following their original configurations, resulting in a larger number of parameters compared to LRX-PINN. This design indicates that the observed performance improvement is mainly associated with the proposed representation under the tested configurations, rather than with increased parameter counts or deeper network stacking.

\paragraph{Cauchy-Enhanced Baselines.}

For the strongly convection-dominated benchmarks Tasks 5--6, our objective is to assess the contribution of the proposed integrated Cauchy representation independently of a particular PINN formulation. To this end, we embed the proposed activation into state-of-the-art hp-VPINN-based frameworks from the literature while keeping their original loss functionals, stabilization strategies, and training settings unchanged. The resulting models are referred to as LRX-enhanced hp-VPINN.

For Task 5, we consider the improved hp-VPINN framework proposed in \cite{anandh2025}. The SUPG-stabilized loss functionals, namely $\mathcal{L}_{\tau}^{\mathrm{SUPG}}$ with constant $\tau$ and $\mathcal{L}_{\tau}^{\mathrm{SUPG}}$ with learned $\tau$, are kept identical to those in the original work, together with the same network configurations and training settings. The only modification is the replacement of the original activation function with the proposed integrated Cauchy activation.

For Task 6, since this benchmark was not considered in \cite{anandh2025}, we use the same LRX-enhanced hp-VPINN setting and choose the network depth, width, and training settings to match those used in \cite{frerichs2026}. The resulting model is compared with the best-performing configuration reported in \cite{frerichs2026}, namely the Limited Residual PINN equipped with the Crosswind Loss Functional.

Through Tasks 5--6, we isolate the effect of the proposed Cauchy-based activation structure from the effects of loss design, stabilization, and training strategy.

\paragraph{Evaluation Metrics.}

We evaluate all methods using the relative $L_2$ error (Rel-$L_2$), local relative error on the layer region (Rel-line), maximum boundary error (Bc-max), PDE residual root mean square error (RMSE) and parameter count (Params).

The global relative error is defined as
\[
\mathrm{Rel}\text{-}L_2=
\frac{\|u_\theta-u\|_{L_2(\Omega)}}
{\|u\|_{L_2(\Omega)}}.
\]

The local relative error on the target layer region $\Omega_{\mathrm{line}}$ is defined as
\[
\mathrm{Rel}\text{-line}=
\frac{\|u_\theta-u\|_{L_2(\Omega_{\mathrm{line}})}}
{\|u\|_{L_2(\Omega_{\mathrm{line}})}}.
\]

The maximum boundary error is
\[
\mathrm{Bc}\text{-max}
=
\max_{\mathbf{x}\in\partial\Omega}
|u_\theta(\mathbf{x})-u(\mathbf{x})|.
\]

The residual root mean square error is
\[
\mathrm{RMSE}
=
\left(
\frac{1}{|\Omega|}
\int_\Omega
|\mathcal{R}(u_\theta)|^2
\,d\mathbf{x}
\right)^{1/2}.
\]

Smaller values of Rel-$L_2$, Rel-line, Bc-max, and RMSE correspond to higher approximation accuracy and better PDE satisfaction.
  
\section{Numerical Results}\label{numerical results}

In this section, we systematically evaluate the proposed LRX-PINN on six representative convection--diffusion benchmarks with increasing layer complexity. The experiments are designed to assess the approximation capability of the proposed Cauchy-based network architecture, the effectiveness of the layer-adapted integrated Cauchy activation in strongly convection-dominated regimes, and the transferability of the proposed activation when embedded into advanced PINN frameworks.

For mildly convection-dominated problems (\textbf{Tasks 1--4} with $\varepsilon=0.5$, $0.1$, and $0.05$), the original and integrated Cauchy activations exhibit nearly identical performance. Therefore, only the results obtained using the original Cauchy activation are reported in these cases. To ensure a fair comparison and to isolate the contribution of the proposed method, the LRX-PINN is compared with the original PIKAN and Fourier PINN architectures without additional problem-specific modifications or optimizations. It is worth noting that both baseline models have been widely reported to achieve strong performance compared with conventional PINN methods on a variety of benchmark problems. Consequently, the performance improvements observed in this work are primarily associated with the proposed Cauchy-based neural representation under the considered experimental settings rather than to differences in training strategies or model tuning.

When convection becomes strongly dominant (\textbf{Tasks 1--4} with $\varepsilon=10^{-2}$ and \textbf{Tasks 5--6}), the integrated Cauchy activation consistently demonstrates significant advantages over the original formulation. This observation provides experimental evidence for the proposed
layer-adapted design and confirms the importance of matching the activation structure to the value-derivative signature of thin layers.

According to our experimental results, the benchmark models achieve satisfactory performance when the parameter $\varepsilon$ is relatively large. Performance degradation is observed only as $\varepsilon$ decreases. This behavior is consistent with the well-known representational limitations of conventional PINN-based methods when dealing with increasingly challenging multiscale or convection-diffusion problems. Therefore, the results indicate that the benchmark methods have already demonstrated their expected performance under appropriate settings. The observed performance gap between LRX-PINN and other methods is not solely attributable to differences in hyperparameter tuning, but rather reflects the intrinsic limitations of the underlying model architectures.

\subsection{Task 1: Smooth Convection-Diffusion Problem}

We consider the two-dimensional convection--diffusion equation
\[
-\Delta u(x,y) + \partial_x u(x,y) = f(x,y),
\qquad (x,y)\in \Omega=(0,1)^2,
\]
with Dirichlet boundary condition
\[
u(x,y)=u_{\mathrm{exact}}(x,y),
\qquad (x,y)\in\partial\Omega.
\]

The exact solution is chosen as
\[
u_{\mathrm{exact}}(x,y)
=
\sin(\pi x)\sin(\pi y).
\]

Hence,
\[
f(x,y)
=
-\Delta u_{\mathrm{exact}} + \partial_x u_{\mathrm{exact}}
=
2\pi^2 \sin(\pi x)\sin(\pi y)
+
\pi \cos(\pi x)\sin(\pi y).
\]
Error analysis and quantitative comparisons are presented in Table \ref{eps1}. Visualizations of the exact and predicted solutions are shown in Figure \ref{task1}.

\begin{table}[H]
\caption{Error comparison of Task 1.}
\centering
\footnotesize
\begin{tabular}{lccccccc}
\toprule
Method 
& Rel-$L_2$
& Bc-max
& RMSE
& Params\\
\midrule

PIKAN
& $9.007{\times}10^{-4}$
& $2.514{\times}10^{-3}$
& $1.325{\times}10^{-2}$
& 7065\\

Fourier PINN
& $3.354{\times}10^{-4}$
& $2.193{\times}10^{-3}$
& $1.780{\times}10^{-3}$
& 3201\\

LRX-PINN
& $\mathbf{6.202{\times}10^{-5}}$ 
& $\mathbf{1.848{\times}10^{-4}}$ 
& $\mathbf{9.865{\times}10^{-4}}$
& 841 \\

\bottomrule
\end{tabular}
\label{eps1}
\end{table}

\begin{figure}[H]
\centering
\includegraphics[width=\textwidth]{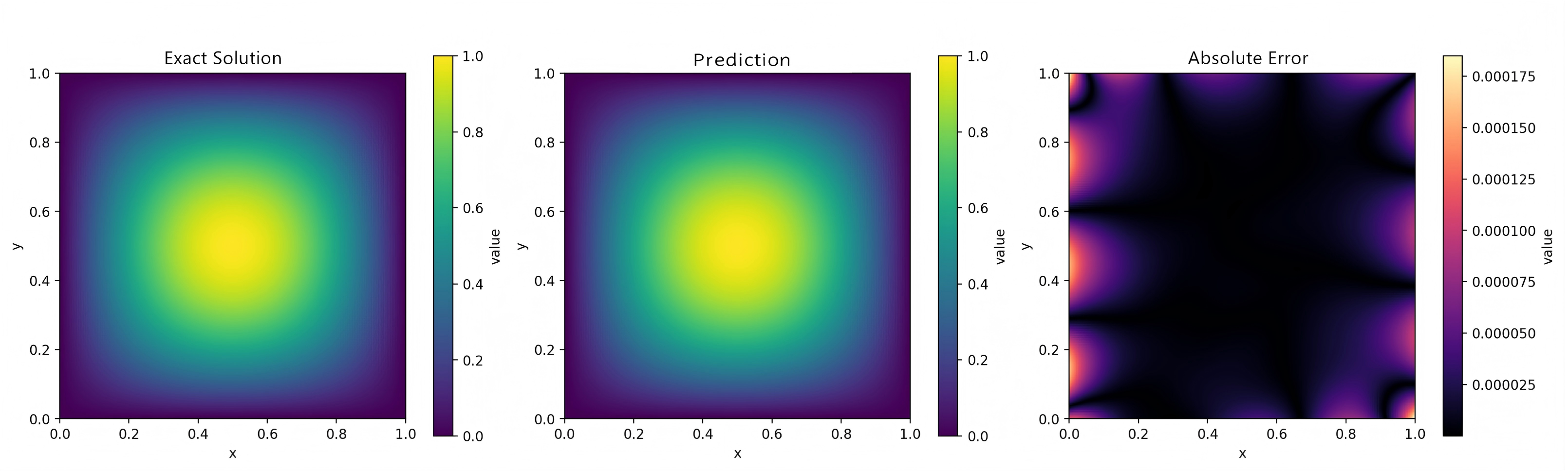}
\caption{Exact solution (left), LRX-PINN prediction (middle), and absolute error (right) for Task 1.}
\label{task1} 
\end{figure}

It should be noted that, for \textbf{Tasks 1-4}, the numbers of trainable parameters for all models are identical to those reported in Table \ref{para}. To avoid redundancy, the parameter counts are omitted from the subsequent tables.

As shown in Table~\ref{eps1}, the proposed LRX-PINN achieves the highest accuracy among all compared methods while using only 841 trainable parameters. The relative $L_2$ error reaches the order of $10^{-5}$, representing an improvement of approximately one order of magnitude over PIKAN and more than five times over Fourier PINN.

Since \textbf{Task~1} does not contain localized layer structures, this experiment primarily evaluates the approximation capability of the network itself. The results indicate that the proposed Cauchy-based architecture does not sacrifice accuracy on smooth solutions despite its strong localization properties. Instead, it provides a highly efficient representation that achieves superior accuracy with substantially fewer parameters than the competing approaches.

\subsection{Task 2: Interior-Layer Problem}

We consider the two-dimensional convection--diffusion equation
\[
-\varepsilon \Delta u(x,y) + \left(x-\frac12\right)\partial_x u(x,y) = f(x,y),
\qquad (x,y)\in(0,1)^2,
\]
with Dirichlet boundary condition
\[
u(x,y)=u_{\mathrm{exact}}(x,y),
\qquad (x,y)\in\partial\Omega,
\quad \Omega=(0,1)^2.
\]

The exact solution is chosen as
\[
u_{\mathrm{exact}}(x,y)
=
\operatorname{tanh}\!\left(\frac{x-\frac12}{\varepsilon}\right).
\]

In this case,
\[
f(x,y)=
\frac{2}{\varepsilon}
\operatorname{sech}^2\!\left(\frac{x-\frac12}{\varepsilon}\right)
\operatorname{tanh}\!\left(\frac{x-\frac12}{\varepsilon}\right)
+
\frac{x-\frac12}{\varepsilon}
\operatorname{sech}^2\!\left(\frac{x-\frac12}{\varepsilon}\right).
\]

The exact solution contains an interior transition layer around \(x=0.5\).
As \(\varepsilon\) decreases, the layer becomes thinner and the gradient
becomes sharper, making accurate numerical approximation increasingly
difficult. This poses a significant challenge for numerical methods due to the presence of sharp gradients within a narrow region. Such behavior is commonly observed in many engineering applications.

To evaluate the performance of LRX-PINN, we conduct experiments across different layer widths corresponding to varying values of \(\varepsilon\). In this problem, we consider $$ \varepsilon \in \left\{5\times10^{-1},\;1\times 10^{-1},\;5\times10^{-2},\;1\times10^{-2}\right\}$$.

\begin{table}[H]
\caption{Error comparison of Task 2 ($\varepsilon = 5\times 10^{-1}$).}
\centering
\footnotesize
\begin{tabular}{lccccccc}
\toprule
Method 
& Rel-$L_2$
& {Rel\text{-}line}
& Bc-max
& RMSE \\
\midrule

PIKAN
& $1.351{\times}10^{-4}$
& $2.147{\times}10^{-4}$
& $3.162{\times}10^{-4}$
& $1.808{\times}10^{-3}$\\

Fourier PINN
& $2.113{\times}10^{-4}$
& $1.827{\times}10^{-4}$
& $2.007{\times}10^{-3}$
& $1.071{\times}10^{-2}$\\

\textbf{LRX-PINN}
& $\mathbf{5.267{\times}10^{-5}}$
& $\mathbf{7.303{\times}10^{-5}}$
& $\mathbf{9.778{\times}10^{-5}}$
& $\mathbf{7.440{\times}10^{-4}}$\\

\bottomrule
\end{tabular}

\label{eps25e-1}
\end{table}

\begin{figure}[H]
\centering
\includegraphics[width=1\textwidth]{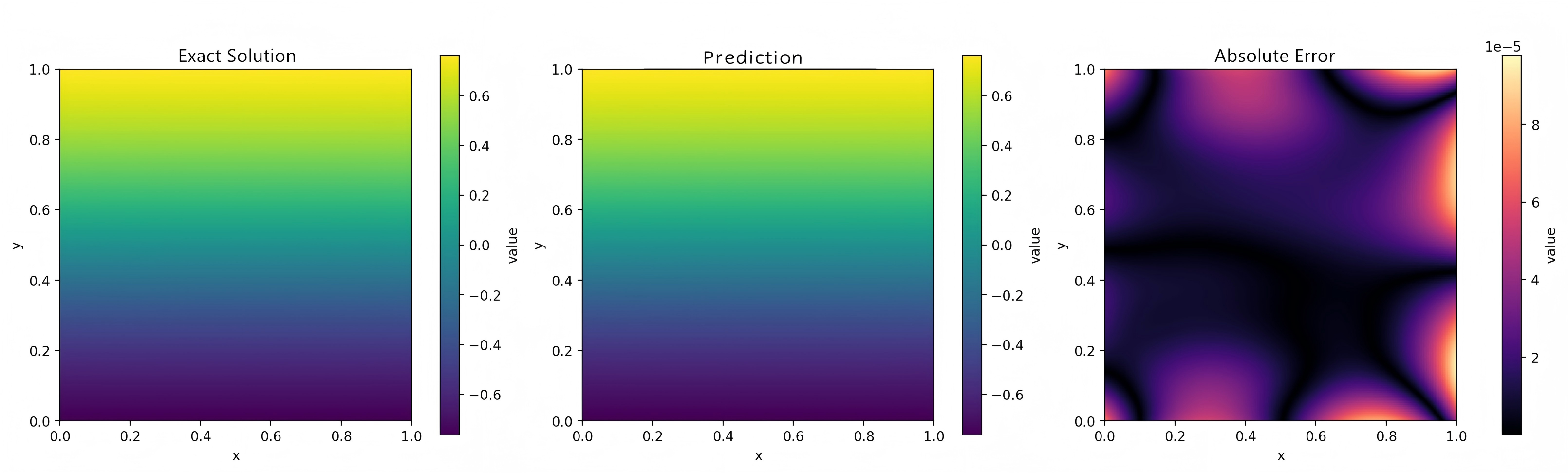}
\caption{Exact solution (left), LRX-PINN prediction (middle), and absolute error (right) for Task 2 ($\varepsilon=5\times 10^{-1}$).}
\label{task25e-1}
\end{figure}

\begin{table}[H]
\caption{Error comparison of Task 2 ($\varepsilon = 1\times 10^{-1}$).}
\centering
\footnotesize
\begin{tabular}{lccccccc}
\toprule
Method 
& Rel-$L_2$
& {Rel\text{-}line}
& Bc-max
& RMSE \\
\midrule

PIKAN
& $3.149{\times}10^{-4}$
& $3.019{\times}10^{-4}$
& $1.429{\times}10^{-3}$
& $8.062{\times}10^{-3}$\\

Fourier PINN
& $5.046{\times}10^{-4}$
& $5.260{\times}10^{-4}$
& $4.970{\times}10^{-3}$
& $2.859{\times}10^{-2}$\\

\textbf{LRX-PINN}
& $\mathbf{6.194{\times}10^{-5}}$
& $\mathbf{5.149{\times}10^{-5}}$
& $\mathbf{8.540{\times}10^{-5}}$
& $\mathbf{2.955{\times}10^{-4}}$\\
\bottomrule
\end{tabular}

\label{eps21e-1}
\end{table}

\begin{figure}[H]
\centering
\includegraphics[width=1\textwidth]{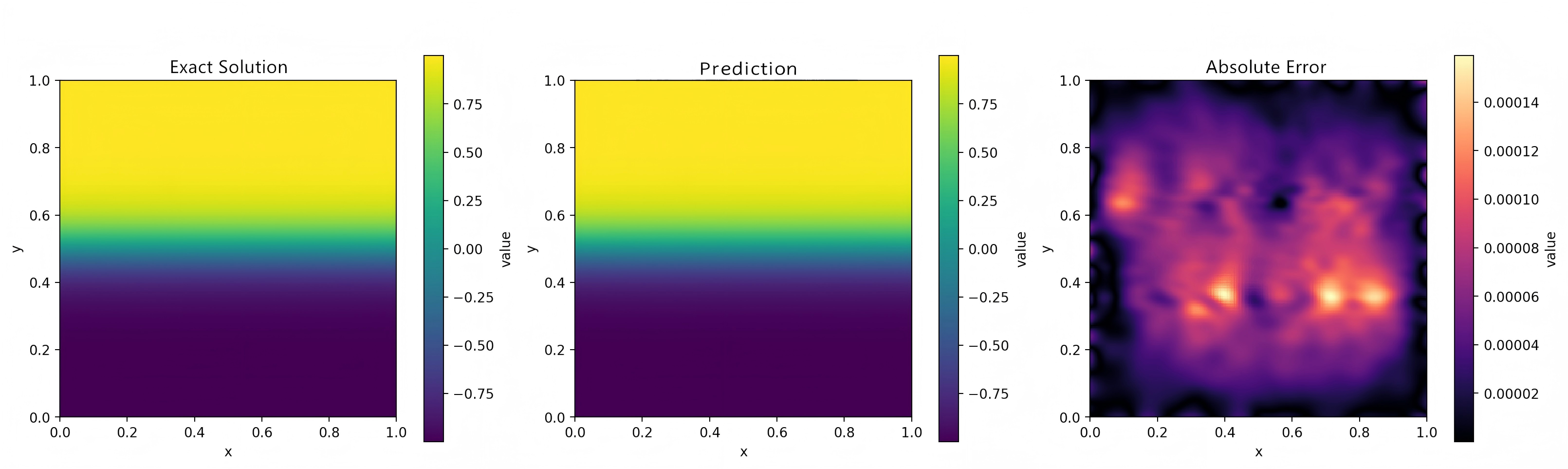}
\caption{Exact solution (left), LRX-PINN prediction (middle), and absolute error (right) for Task 2 ($\varepsilon=1\times 10^{-1}$).}
\label{task21e-1}
\end{figure}

\begin{table}[H]
\caption{Error comparison of Task 2 ($\varepsilon = 5\times 10^{-2}$).}
\centering
\footnotesize
\begin{tabular}{lccccccc}
\toprule
Method 
& Rel-$L_2$
& {Rel\text{-}line}
& Bc-max
& RMSE \\
\midrule

PIKAN
& $5.623{\times}10^{-3}$
& $6.293{\times}10^{-3}$
& $3.997{\times}10^{-3}$
& $2.828{\times}10^{-2}$\\

Fourier PINN
& $1.455{\times}10^{-3}$
& $1.535{\times}10^{-3}$
& $1.114{\times}10^{-2}$
& $8.425{\times}10^{-2}$\\

\textbf{LRX-PINN}
& $\mathbf{2.022{\times}10^{-4}}$
& $\mathbf{5.345{\times}10^{-5}}$
& $\mathbf{1.323{\times}10^{-4}}$
& $\mathbf{4.171{\times}10^{-3}}$\\

\bottomrule
\end{tabular}

\label{eps25e-2}
\end{table}

\begin{figure}[H]
\centering
\includegraphics[width=1\textwidth]{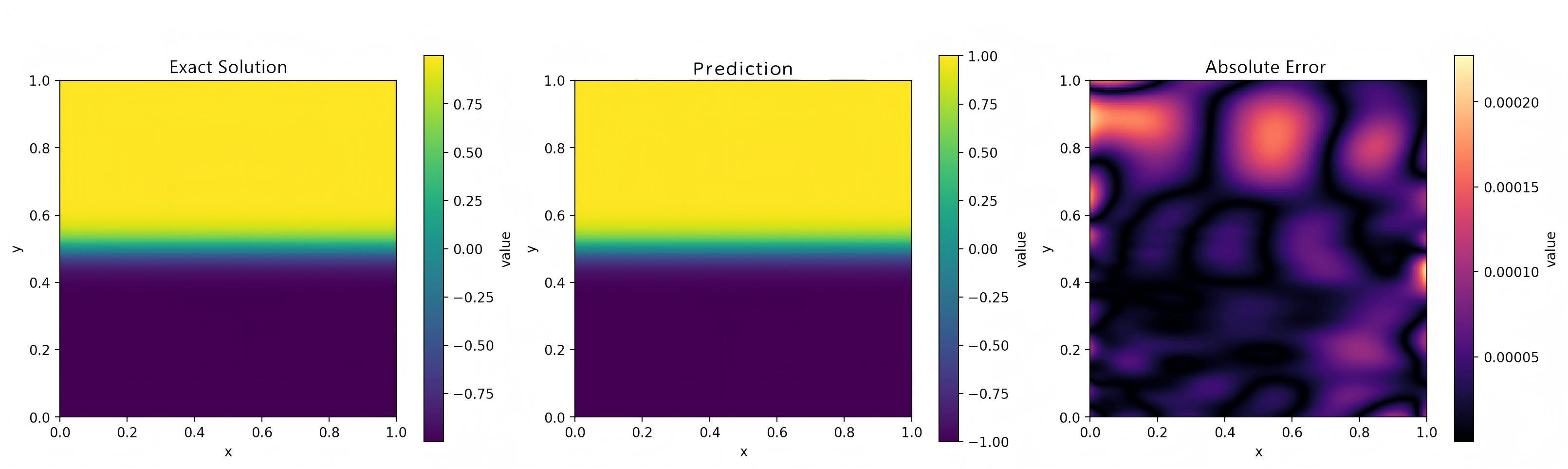}
\caption{Exact solution (left), LRX-PINN prediction (middle), and absolute error (right) for Task 2 ($\varepsilon=5\times 10^{-2}$).}
\label{task25e-2}
\end{figure}

Tables~\ref{eps25e-1}--\ref{eps25e-2} demonstrate that the proposed LRX-PINN consistently achieves the best performance across all moderately convection-dominated cases. As the diffusion parameter decreases, the interior layer becomes progressively thinner and the approximation task becomes more challenging. While both PIKAN and Fourier PINN exhibit noticeable accuracy degradation, the proposed method maintains stable approximation quality and preserves local layer resolution.

In particular, even for $\varepsilon=5\times10^{-2}$, the relative $L_2$ error remains at the order of $10^{-4}$ and the local Rel-line error remains at the order of $10^{-5}$. These results indicate that the original Cauchy activation already provides an effective localized representation for moderately thin interior layers.

\begin{table}[H]
\caption{Error comparison of Task 2 ($\varepsilon = 1\times 10^{-2}$).}
\centering
\footnotesize
\begin{tabular}{lcccccccc}
\toprule
Method 
& Rel-$L_2$
& {Rel\text{-}line}
& Bc-max
& RMSE \\
\midrule

PIKAN
& $4.854{\times}10^{-2}$
& $8.980{\times}10^{-2}$
& $8.343{\times}10^{-2}$
& $2.276{\times}10^{-1}$\\

Fourier PINN
& $5.782{\times}10^{-2}$
& $6.718{\times}10^{-2}$
& $5.449{\times}10^{-2}$
& $3.856{\times}10^{-1}$\\

LRX-PINN (Cauchy activation)
& $1.714{\times}10^{-3}$
& $1.008{\times}10^{-3}$
& $2.366{\times}10^{-3}$
& $5.399{\times}10^{-3}$\\

\textbf{LRX-PINN (integrated Cauchy activation)}
& $\mathbf{2.668{\times}10^{-4}}$
& $\mathbf{5.369{\times}10^{-4}}$
& $\mathbf{1.708{\times}10^{-4}}$
& $\mathbf{2.046{\times}10^{-3}}$\\

\bottomrule
\end{tabular}
\label{eps21e-2}
\end{table}

\begin{figure}[H]
\centering
\includegraphics[width=1\textwidth]{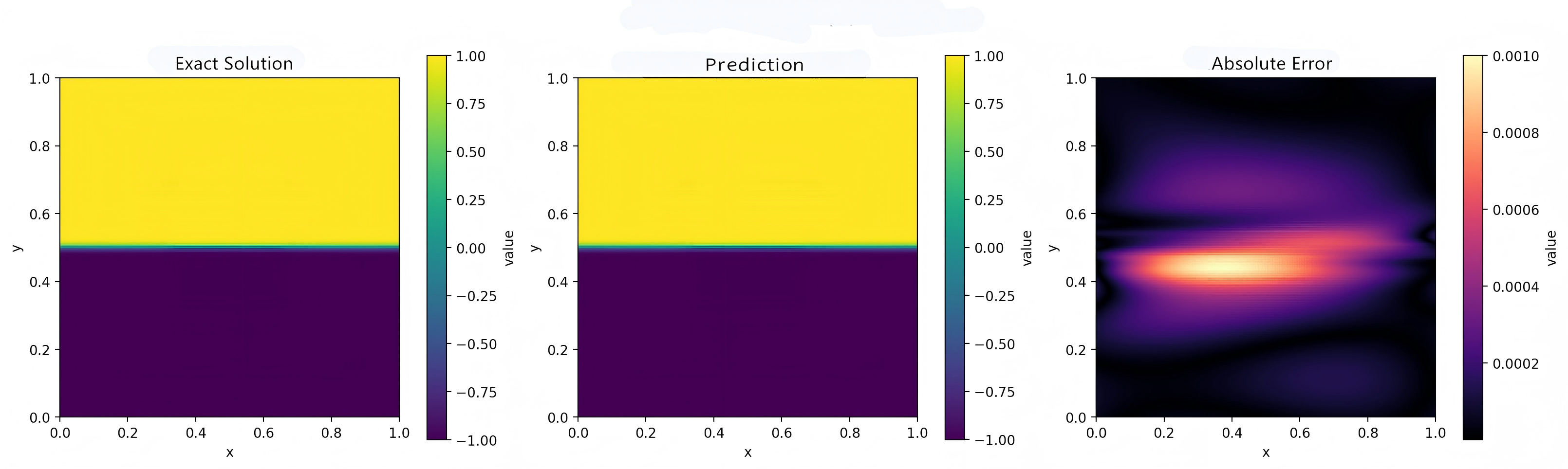}
\caption{Exact solution (left), LRX-PINN prediction (middle), and absolute error (right) for Task 2 ($\varepsilon=1\times 10^{-2}$).}
\label{task21e-2}
\end{figure}
Table~\ref{eps21e-2} corresponds to the convection-dominated regime with $\varepsilon=10^{-2}$, where the interior-layer width becomes extremely small and the solution exhibits highly localized transition behavior. In this setting, both PIKAN and Fourier PINN experience substantial degradation, yielding relative errors of order $10^{-2}$ and failing to accurately resolve the sharp interior layer.

Although the original Cauchy activation already provides a significant improvement over the competing approaches, the proposed integrated Cauchy activation further reduces the relative $L_2$ error from $1.714\times10^{-3}$ to $2.668\times10^{-4}$ and improves the boundary and residual errors by nearly one order of magnitude. This observation confirms that, when the layer thickness approaches the intrinsic approximation scale of the network, incorporating a transition-type activation whose effective width is aligned with the layer width becomes crucial for accurately representing thin interior layers.

The substantial reduction in both Rel-$L_2$ and Rel-line demonstrates that the proposed layer-resolving activation not only improves global accuracy but also significantly enhances the resolution of highly localized structures.

\subsection{Task 3: Boundary-Layer Problem}

We consider the two-dimensional convection-diffusion equation
\[
-\varepsilon \Delta u(x,y) + \partial_x u(x,y) = f(x,y),
\qquad (x,y)\in \Omega=(0,1)^2,
\]
with Dirichlet boundary condition
\[
u(x,y)=u_{\mathrm{exact}}(x,y),
\qquad (x,y)\in\partial\Omega.
\]

The exact solution is chosen as
\[
u_{\mathrm{exact}}(x,y)
=
\left(1-\exp\!\left(-\frac{1-x}{\varepsilon}\right)\right)\sin(\pi y).
\]

Hence,
\[
f(x,y)
=
-\varepsilon \Delta u_{\mathrm{exact}} + \partial_x u_{\mathrm{exact}}
=
\varepsilon \pi^2
\left(1-\exp\!\left(-\frac{1-x}{\varepsilon}\right)\right)\sin(\pi y).
\]

We still consider $\varepsilon \in \left\{5\times10^{-1},\;1\times 10^{-1},\;5\times10^{-2},\;1\times10^{-2}\right\}$.

\begin{table}[H]
\caption{Error comparison of Task 3 ($\varepsilon = 5\times 10^{-1}$).}
\centering
\footnotesize
\begin{tabular}{lccccccc}
\toprule
Method 
& Rel-$L_2$
& {Rel\text{-}line}
& Bc-max
& RMSE \\
\midrule

PIKAN
& $5.333{\times}10^{-3}$
& $2.497{\times}10^{-2}$
& $2.086{\times}10^{-2}$
& $3.818{\times}10^{-2}$\\

Fourier PINN
& $1.720{\times}10^{-4}$
& $4.364{\times}10^{-4}$
& $5.937{\times}10^{-4}$
& $4.648{\times}10^{-3}$\\

\textbf{LRX-PINN}
& $\mathbf{3.360{\times}10^{-5}}$
& $\mathbf{2.807{\times}10^{-5}}$
& $\mathbf{1.678{\times}10^{-4}}$
& $\mathbf{1.747{\times}10^{-4}}$\\

\bottomrule
\end{tabular}
\label{eps35e-1}
\end{table}

\begin{figure}[H]
\centering
\includegraphics[width=1\textwidth]{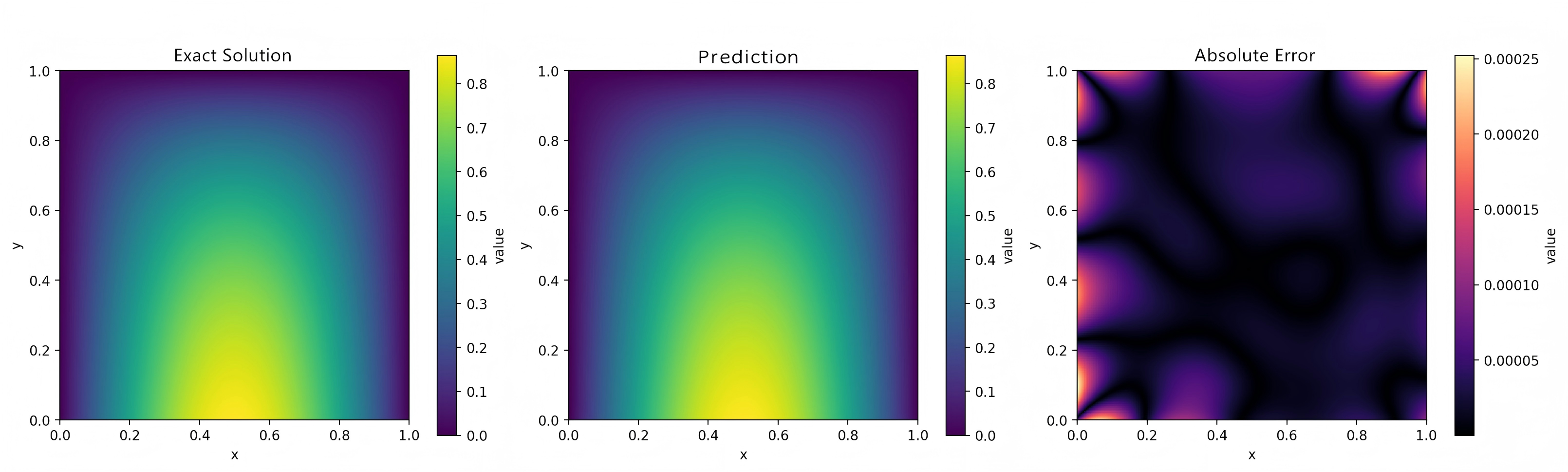}
\caption{Exact solution (left), LRX-PINN prediction (middle), and absolute error (right) for Task 3 ($\varepsilon=5\times 10^{-1}$).}
\label{task35e-1}
\end{figure}

\begin{table}[H]
\caption{Error comparison of Task 3 ($\varepsilon = 1\times 10^{-1}$).}
\centering
\footnotesize
\begin{tabular}{lccccccc}
\toprule
Method 
& Rel-$L_2$
& {Rel\text{-}line}
& Bc-max
& RMSE \\
\midrule

PIKAN
& $4.064{\times}10^{-3}$
& $9.560{\times}10^{-3}$
& $3.570{\times}10^{-2}$
& $2.289{\times}10^{-2}$\\

Fourier PINN
& $5.715{\times}10^{-3}$
& $6.986{\times}10^{-3}$
& $1.744{\times}10^{-2}$
& $1.253{\times}10^{-1}$\\

\textbf{LRX-PINN}
& $\mathbf{2.018{\times}10^{-5}}$
& $\mathbf{7.357{\times}10^{-6}}$
& $\mathbf{1.591{\times}10^{-4}}$
& $\mathbf{6.419{\times}10^{-4}}$\\

\bottomrule
\end{tabular}

\label{eps31e-1}
\end{table}

\begin{figure}[H]
\centering
\includegraphics[width=1\textwidth]{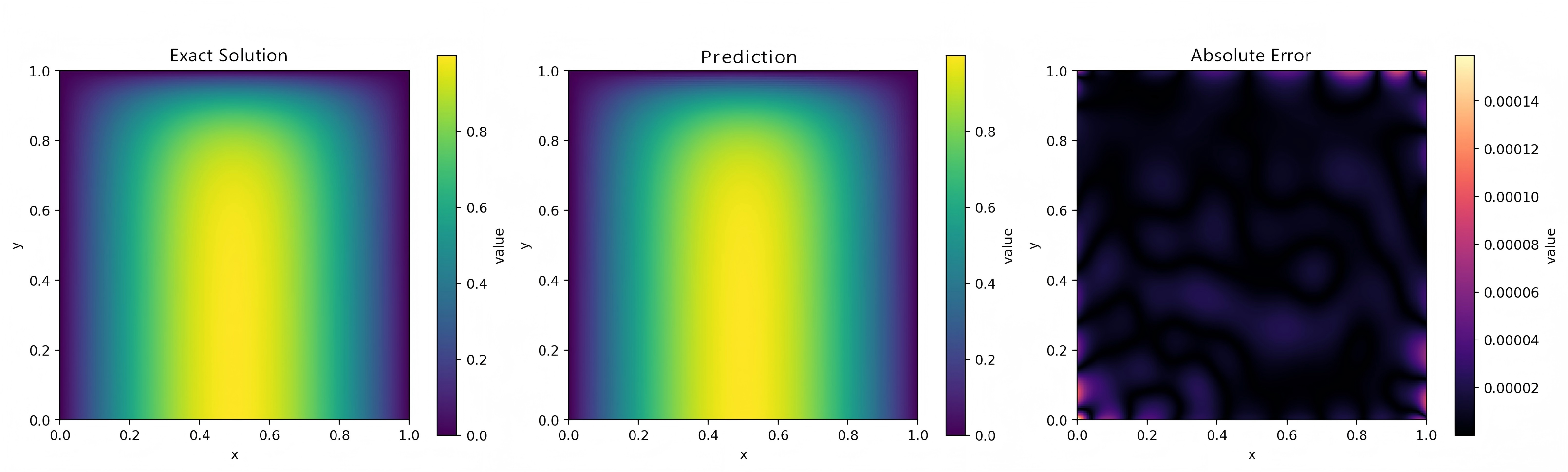}
\caption{Exact solution (left), LRX-PINN prediction (middle), and absolute error (right) for Task 3 ($\varepsilon=1\times 10^{-1}$).}
\label{task31e-1}
\end{figure}

\begin{table}[H]
\caption{Error comparison of Task 3 ($\varepsilon = 5\times 10^{-2}$).}
\centering
\footnotesize
\begin{tabular}{lccccccc}
\toprule
Method 
& Rel-$L_2$
& {Rel\text{-}line}
& Bc-max
& RMSE \\
\midrule

PIKAN
& $7.916{\times}10^{-2}$
& $1.153{\times}10^{-1}$
& $1.235{\times}10^{-1}$
& $1.397{\times}10^{-1}$\\

Fourier PINN
& $6.445{\times}10^{-3}$
& $8.477{\times}10^{-3}$
& $1.316{\times}10^{-2}$
& $1.245{\times}10^{-1}$\\

\textbf{LRX-PINN}
& $\mathbf{5.444{\times}10^{-5}}$
& $\mathbf{3.126{\times}10^{-5}}$
& $\mathbf{4.056{\times}10^{-4}}$
& $\mathbf{1.459{\times}10^{-3}}$\\

\bottomrule
\end{tabular}

\label{eps35e-2}
\end{table}

\begin{figure}[H]
\centering
\includegraphics[width=1\textwidth]{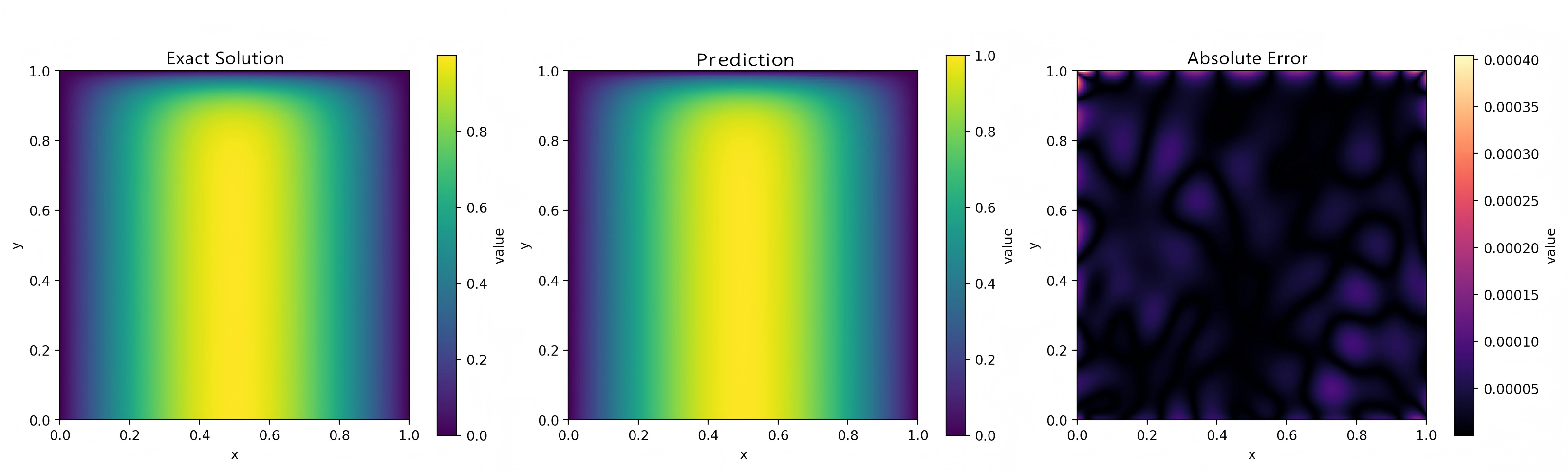}
\caption{Exact solution (left), LRX-PINN prediction (middle), and absolute error (right) for Task 3 ($\varepsilon=5\times 10^{-2}$).}
\label{task35e-2}
\end{figure}

Tables~\ref{eps35e-1}--\ref{eps35e-2} show that the proposed method remains accurate as the boundary layer becomes increasingly thin. While the errors of PIKAN and Fourier PINN grow significantly as $\varepsilon$ decreases, LRX-PINN maintains relative errors at the order of $10^{-5}$ across all tested cases.

This behavior indicates that the Cauchy-based representation naturally captures the highly localized structure of boundary layers. Consequently, the Cauchy-based representation remains relatively robust as
the layer width decreases.

\begin{table}[H]
\caption{Error comparison of Task 3 ($\varepsilon = 1\times 10^{-2}$).}
\centering
\footnotesize
\begin{tabular}{lcccccccc}
\toprule
Method 
& Rel-$L_2$
& {Rel\text{-}line}
& Bc-max
& RMSE \\
\midrule

PIKAN
& $6.444{\times}10^{-1}$
& $9.802{\times}10^{-1}$
& $1.999{\times}10^{-2}$
& $7.839{\times}10^{-1}$\\

Fourier PINN
& $7.881{\times}10^{-1}$
& $9.976{\times}10^{-1}$
& $1.235{\times}10^{-2}$
& $2.233{\times}10^{0}$\\

LRX-PINN (Cauchy activation)
& $1.471{\times}10^{-1}$
& $2.283{\times}10^{-1}$
& $1.210{\times}10^{-2}$
& $2.433{\times}10^{-1}$\\

\textbf{LRX-PINN (integrated Cauchy activation)}
& $\mathbf{4.645{\times}10^{-3}}$
& $\mathbf{6.867{\times}10^{-3}}$
& $\mathbf{1.743{\times}10^{-3}}$
& $\mathbf{2.237{\times}10^{-2}}$\\

\bottomrule
\end{tabular}

\label{eps31e-2}
\end{table}

\begin{figure}[H]
\centering
\includegraphics[width=1\textwidth]{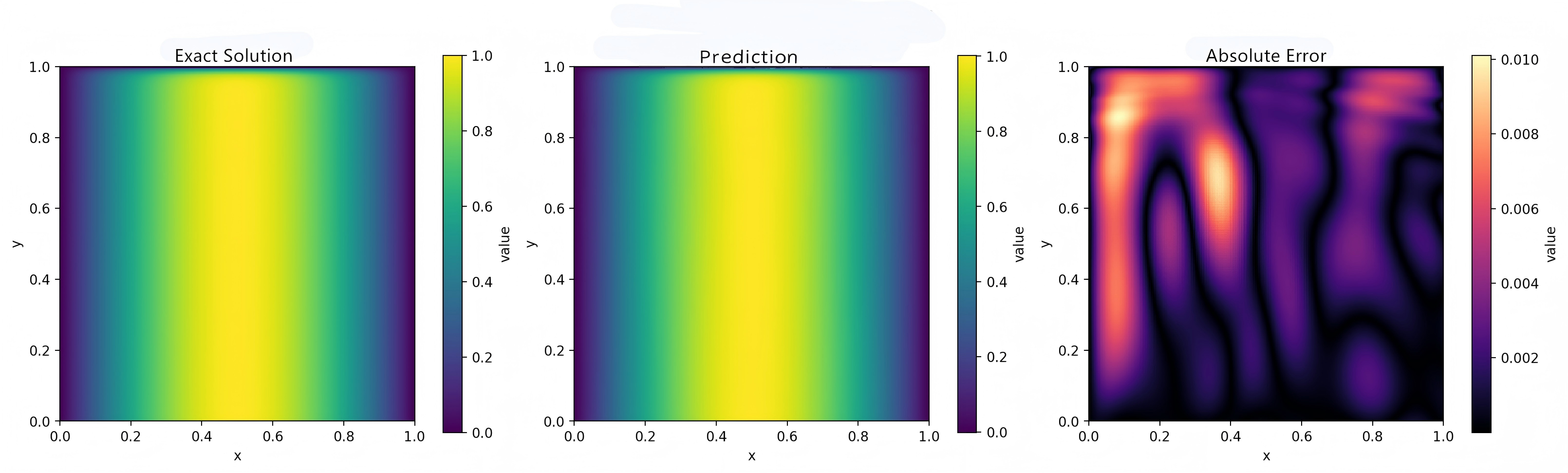}
\caption{Exact solution (left), LRX-PINN prediction (middle), and absolute error (right) for Task 3 ($\varepsilon=1\times 10^{-2}$).}
\label{task31e-2}
\end{figure}

Table~\ref{eps31e-2} presents a considerably more challenging boundary-layer problem with $\varepsilon=10^{-2}$. In this regime, the layer thickness becomes extremely small and standard PINN architectures struggle to capture the sharp boundary transition. As a result, both PIKAN and Fourier PINN exhibit severe performance deterioration, with relative errors approaching unity.

The original Cauchy activation remains capable of capturing the boundary-layer structure to some extent; however, its relative error is still at the order of $10^{-1}$. By contrast, the integrated Cauchy activation reduces the relative error to $4.645\times10^{-3}$ and decreases the residual RMSE by more than one order of magnitude. Such a dramatic improvement highlights the importance of explicitly incorporating layer-adapted structures into the activation design.

These results indicate that, for extremely thin boundary layers, the proposed integrated activation provides a significantly more suitable representation than both conventional neural activations and the original Cauchy formulation.

\subsection{Task 4: Mixed Interior-Boundary Layer Problem}

We consider the two-dimensional convection-diffusion equation
\[
-\varepsilon \Delta u(x,y) + \left(x-\frac12\right)\partial_x u(x,y) = f(x,y),
\qquad (x,y)\in \Omega=(0,1)^2,
\]
with Dirichlet boundary condition
\[
u(x,y)=u_{\mathrm{exact}}(x,y),
\qquad (x,y)\in\partial\Omega.
\]

The exact solution is chosen as a combination of a boundary layer and an interior layer:
\[
u_{\mathrm{exact}}(x,y)
=
\left(1-\exp\!\left(-\frac{1-x}{\varepsilon}\right)\right)
+
\tanh\!\left(\frac{x-\frac12}{\varepsilon}\right).
\]
This problem is more challenging because the solution simultaneously contains an interior transition layer and a boundary layer. The coexistence of multiple localized structures significantly increases the difficulty of accurately approximating the global solution profile.

We take $\varepsilon \in \left\{5\times10^{-1},\;1\times 10^{-1},\;5\times10^{-2},\;1\times10^{-2}\right\}$.

\begin{table}[H]
\caption{Error comparison of Task 4 ($\varepsilon = 5\times 10^{-1}$).}
\centering
\footnotesize
\begin{tabular}{lcccccccc}
\toprule
Method 
& Rel-$L_2$
& {Rel\text{-}line}
& Bc-max
& RMSE \\
\midrule

PIKAN
& $1.879{\times}10^{-4}$
& $1.305{\times}10^{-4}$
& $1.653{\times}10^{-3}$
& $3.724{\times}10^{-3}$\\

Fourier PINN
& $8.656{\times}10^{-4}$
& $4.449{\times}10^{-4}$
& $3.633{\times}10^{-3}$
& $4.389{\times}10^{-2}$\\

\textbf{LRX-PINN}
& $\mathbf{3.731{\times}10^{-5}}$
& $\mathbf{3.196{\times}10^{-5}}$
& $\mathbf{2.221{\times}10^{-4}}$
& $\mathbf{1.463{\times}10^{-3}}$\\

\bottomrule
\end{tabular}

\label{eps45e-1}
\end{table}

\begin{figure}[H]
\centering
\includegraphics[width=1\textwidth]{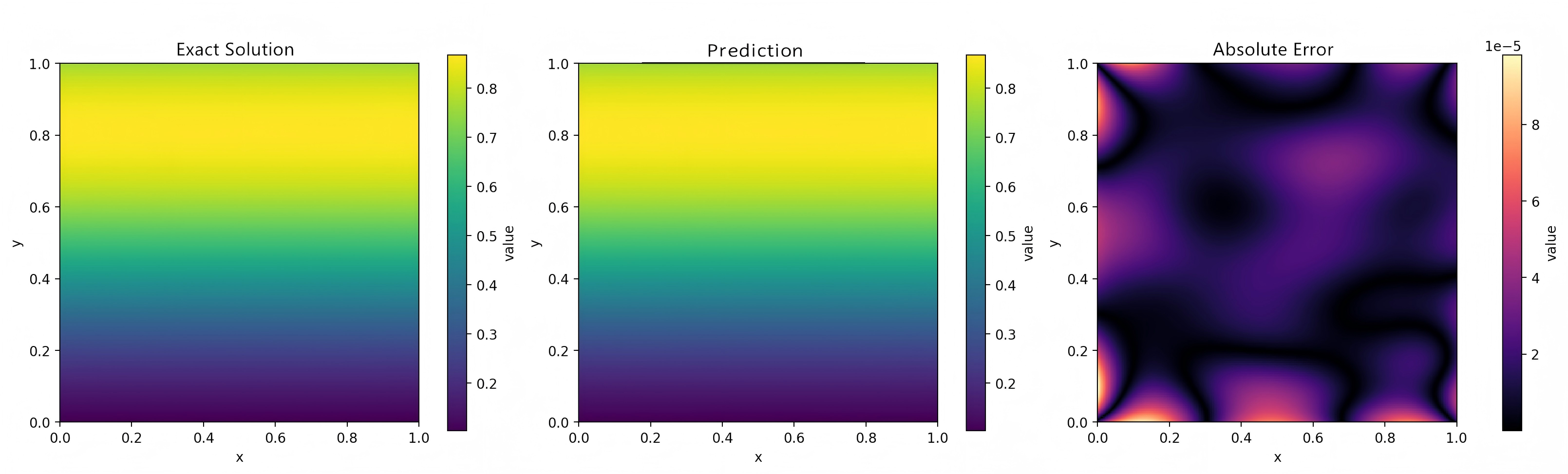}
\caption{Exact solution (left), LRX-PINN prediction (middle), and absolute error (right) for Task 4 ($\varepsilon=5\times 10^{-1}$).}
\label{task45e-1}
\end{figure}

\begin{table}[H]
\caption{Error comparison of Task 4 ($\varepsilon = 1\times 10^{-1}$).}
\centering
\footnotesize
\begin{tabular}{lcccccccc}
\toprule
Method 
& Rel-$L_2$
& {Rel\text{-}line}
& Bc-max
& RMSE \\
\midrule

PIKAN
& $6.224{\times}10^{-4}$
& $6.240{\times}10^{-4}$
& $2.598{\times}10^{-3}$
& $1.334{\times}10^{-2}$\\

Fourier PINN
& $4.478{\times}10^{-4}$
& $4.257{\times}10^{-4}$
& $3.633{\times}10^{-3}$
& $2.798{\times}10^{-2}$\\

\textbf{LRX-PINN}
& $\mathbf{4.847{\times}10^{-5}}$
& $\mathbf{4.096{\times}10^{-5}}$
& $\mathbf{2.076{\times}10^{-4}}$
& $\mathbf{1.460{\times}10^{-3}}$\\

\bottomrule
\end{tabular}

\label{eps41e-1}
\end{table}

\begin{figure}[H]
\centering
\includegraphics[width=1\textwidth]{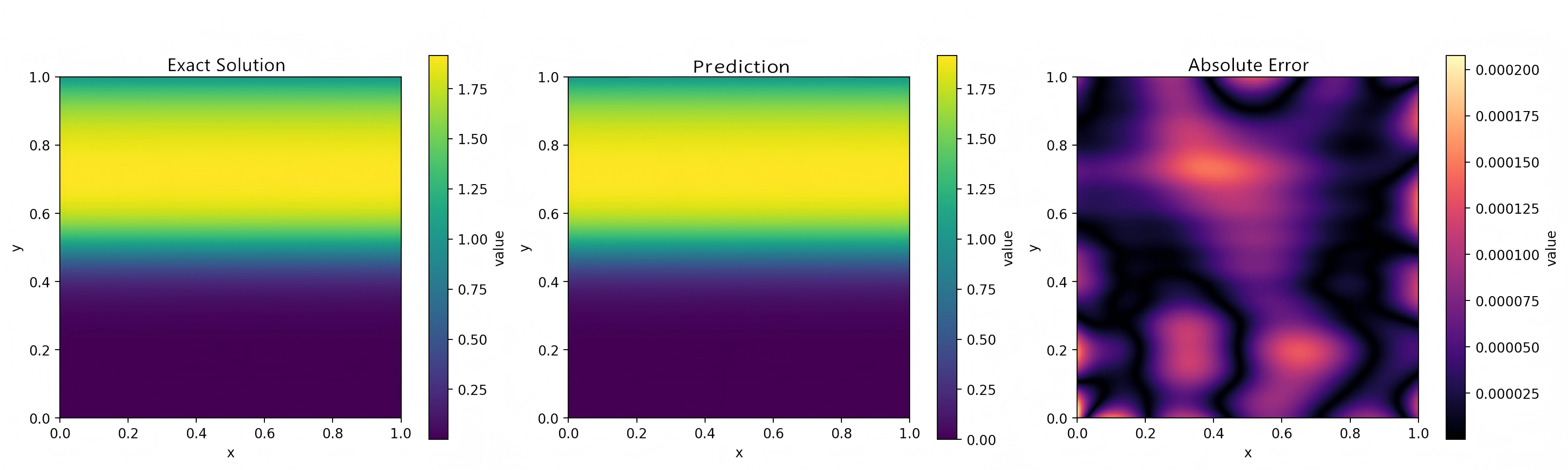}
\caption{Exact solution (left), LRX-PINN prediction (middle), and absolute error (right) for Task 4 ($\varepsilon=1\times 10^{-1}$).}
\label{task41e-1}
\end{figure}

\begin{table}[H]
\caption{Error comparison of Task 4 ($\varepsilon = 5\times 10^{-2}$).}
\centering
\footnotesize
\begin{tabular}{lcccccccc}
\toprule
Method 
& Rel-$L_2$
& {Rel\text{-}line}
& Bc-max
& RMSE \\
\midrule

PIKAN
& $1.834{\times}10^{-2}$
& $2.057{\times}10^{-2}$
& $5.663{\times}10^{-3}$
& $5.154{\times}10^{-2}$\\

Fourier PINN
& $9.711{\times}10^{-2}$
& $6.039{\times}10^{-2}$
& $3.651{\times}10^{-2}$
& $7.947{\times}10^{-1}$\\

\textbf{LRX-PINN}
& $\mathbf{2.090{\times}10^{-4}}$
& $\mathbf{2.186{\times}10^{-4}}$
& $\mathbf{4.285{\times}10^{-4}}$
& $\mathbf{4.147{\times}10^{-3}}$\\

\bottomrule
\end{tabular}

\label{eps45e-2}
\end{table}

\begin{figure}[H]
\centering
\includegraphics[width=1\textwidth]{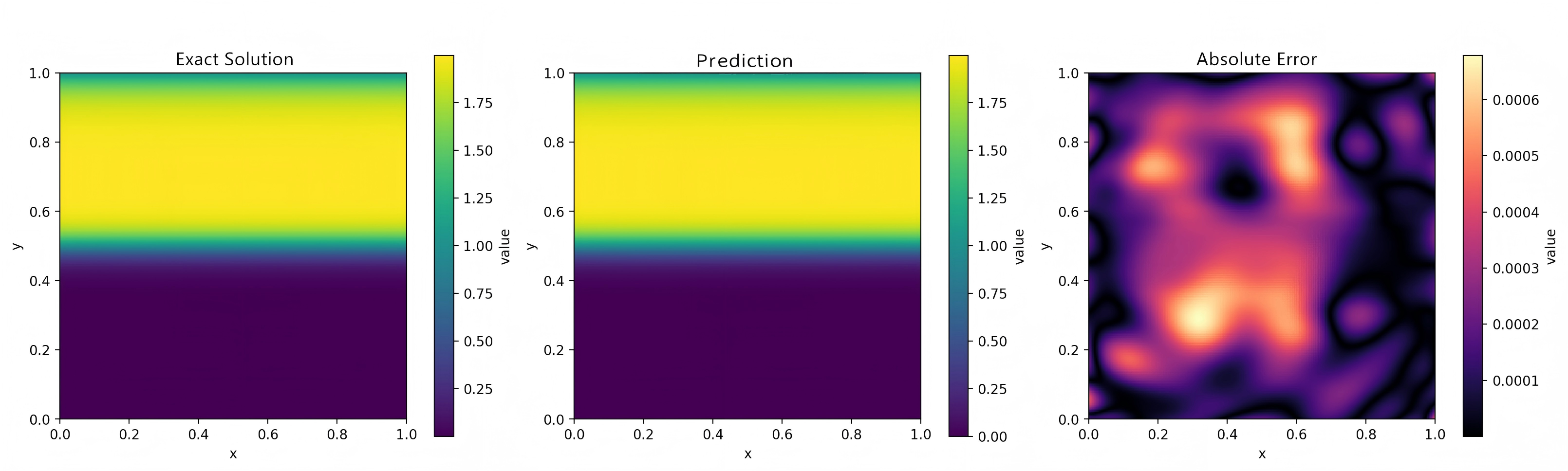}
\caption{Exact solution (left), LRX-PINN prediction (middle), and absolute error (right) for Task 4 ($\varepsilon=5\times 10^{-2}$).}
\label{task45e-2}
\end{figure}

Tables~\ref{eps45e-1} and~\ref{eps41e-1} show that the proposed LRX-PINN maintains excellent approximation accuracy for mixed-layer problems containing both interior and boundary layers. Compared with PIKAN and Fourier PINN, the proposed method consistently achieves the smallest global and local errors while using substantially fewer trainable parameters.

Unlike \textbf{Tasks~2} and \textbf{Tasks~3}, the present benchmark requires the network to simultaneously resolve two different types of localized structures. Nevertheless, the proposed Cauchy-based architecture accurately captures both the interior transition layer and the boundary layer, indicating strong representational capability for multiscale solution profiles.

As the diffusion parameter decreases to $\varepsilon=5\times10^{-2}$, the problem becomes considerably more challenging due to the interaction of increasingly thin interior and boundary layers. As shown in Table~\ref{eps45e-2}, LRX-PINN still achieves a relative $L_2$ error of $2.090\times10^{-4}$, whereas the errors of PIKAN and Fourier PINN increase to the orders of $10^{-2}$ and $10^{-1}$, respectively. This result demonstrates the robustness of the proposed architecture even in the presence of multiple thin localized structures.

\begin{table}[H]
\caption{Error comparison of Task 4 ($\varepsilon = 1\times 10^{-2}$).}
\centering
\footnotesize
\begin{tabular}{lcccccccc}
\toprule
Method 
& Rel-$L_2$
& {Rel\text{-}line}
& Bc-max
& RMSE \\
\midrule

PIKAN
& $1.225{\times}10^{-1}$
& $1.893{\times}10^{-1}$
& $1.363{\times}10^{-2}$
& $2.818{\times}10^{-1}$\\

Fourier PINN
& $7.104{\times}10^{-1}$
& $3.041{\times}10^{-1}$
& $1.498{\times}10^{-1}$
& $1.217{\times}10^{-1}$\\

LRX-PINN (Cauchy activation)
& $5.476{\times}10^{-2}$
& $8.298{\times}10^{-2}$
& $3.912{\times}10^{-3}$
& $2.480{\times}10^{-1}$\\

\textbf{LRX-PINN (integrated Cauchy activation)}
& $\mathbf{8.295{\times}10^{-3}}$
& $\mathbf{3.940{\times}10^{-3}}$
& $\mathbf{5.364{\times}10^{-3}}$
& $\mathbf{1.103{\times}10^{-2}}$\\

\bottomrule
\end{tabular}

\label{eps41e-2}
\end{table}

\begin{figure}[H]
\centering
\includegraphics[width=1\textwidth]{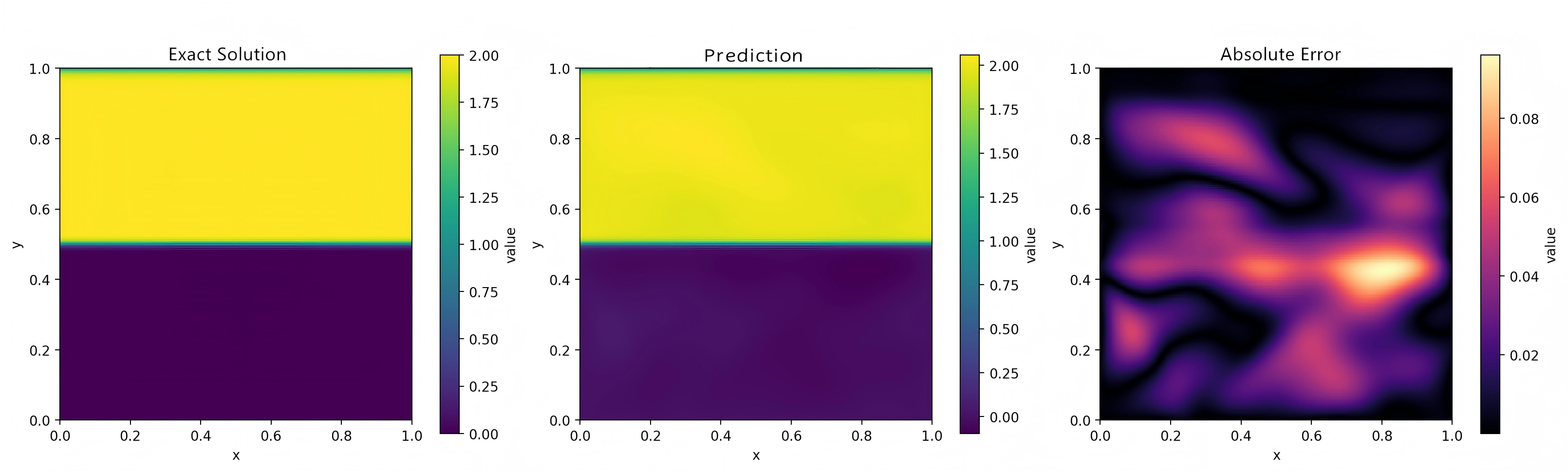}
\caption{Exact solution (left), LRX-PINN prediction (middle), and absolute error (right) for Task 4 ($\varepsilon=1\times 10^{-2}$).}
\label{task41e-2}
\end{figure}

The strongly convection-dominated case with $\varepsilon=10^{-2}$ further highlights the importance of the proposed layer-resolving activation. As shown in Table~\ref{eps41e-2}, both PIKAN and Fourier PINN suffer substantial accuracy degradation in this multiscale setting. Although the original Cauchy activation already outperforms the competing approaches, the proposed integrated Cauchy activation further reduces the relative $L_2$ error from $5.476\times10^{-2}$ to $8.295\times10^{-3}$ and decreases the residual RMSE by more than one order of magnitude.

Compared with \textbf{Tasks~2} and \textbf{Tasks~3}, this benchmark is particularly challenging because multiple thin layers coexist simultaneously. The significant improvement obtained by the integrated activation demonstrates that the proposed layer-resolving transition structure remains effective even when several localized features interact within the same solution. This robustness is particularly important for practical convection-dominated applications, where complex multiscale layer structures frequently arise.

\subsection{Task 5: Outflow Layer Problem}
This benchmark problem was originally proposed by John et al.~\cite{john1997} and has subsequently been widely used for evaluating numerical methods for strongly convection-dominated convection--diffusion equations \cite{anandh2025,frerichs2026}. The problem features two outflow boundary layers whose widths are of order $O(\varepsilon)$ according to asymptotic analysis.

\begin{equation*}
-\varepsilon \Delta u + 2\partial_x u +3\partial_y u+u= f,
\qquad \mathbf{(x,y)}\in(0,1)^2,
\end{equation*}
where $\varepsilon=10^{-8}$, the exact solution has the form:
\[
u(x,y)=xy^2-y^2e^{\frac{2(x-1)}{\varepsilon}}-xe^{\frac{3(y-1)}{\varepsilon}}+e^{\frac{2(x-1)+3(y-1)}{\varepsilon}}
\]
Dirichlet boundary conditions are satisfied. By construction, the outflow boundary layers are located near \(x=1\) and \(y=1\), with a corner-layer interaction near the upper-right corner of the domain. 

\begin{table}[H]
\caption{$L^2$ error for the outflow-layer problem using different PINN-based methods.}
\centering
\footnotesize
\begin{tabular}{lcccc}
\toprule
Method
& $L_{t_0}^{lr}$ loss
& \\
\midrule
$\text{PINNs (various loss functionals, \text{best reported result in \cite{frerichs2026}})}$
& $3.419\times 10^{-2}$\\
\midrule

& $L_\tau^{SUPG}$ constant $\tau$ 
& $L_\tau^{SUPG}$ learnt $\tau$\\
\midrule

$\text{Improved hp-VPINN (\text{best reported result in \cite{anandh2025}})}$
& $1.340\times 10^{-4}$
& $1.037\times 10^{-4}$\\

$\text{LRX-enhanced hp-VPINN (Cauchy activation)}$
& $1.122\times 10^{-4} $
& $1.005\times 10^{-4} $ \\

$\textbf{LRX-enhanced hp-VPINN (integrated Cauchy activation)}$
& $\mathbf{6.318\times 10^{-5}}$
& $\mathbf{3.543\times 10^{-5}}$  \\

\bottomrule
\end{tabular}
\label{d}
\end{table}

\begin{figure}[H]
\centering
\includegraphics[width=0.32\textwidth]{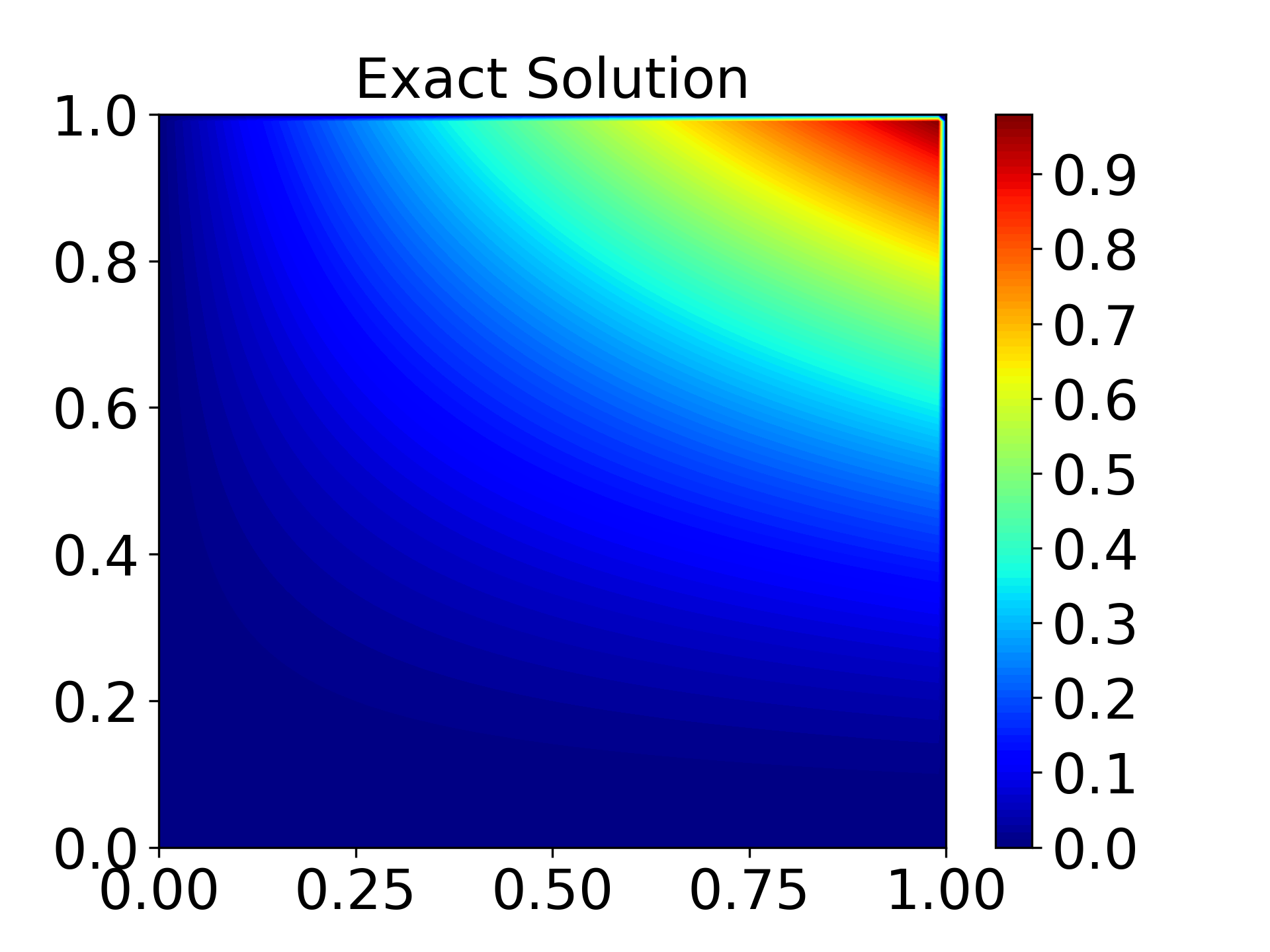}
\hfill
\includegraphics[width=0.32\textwidth]{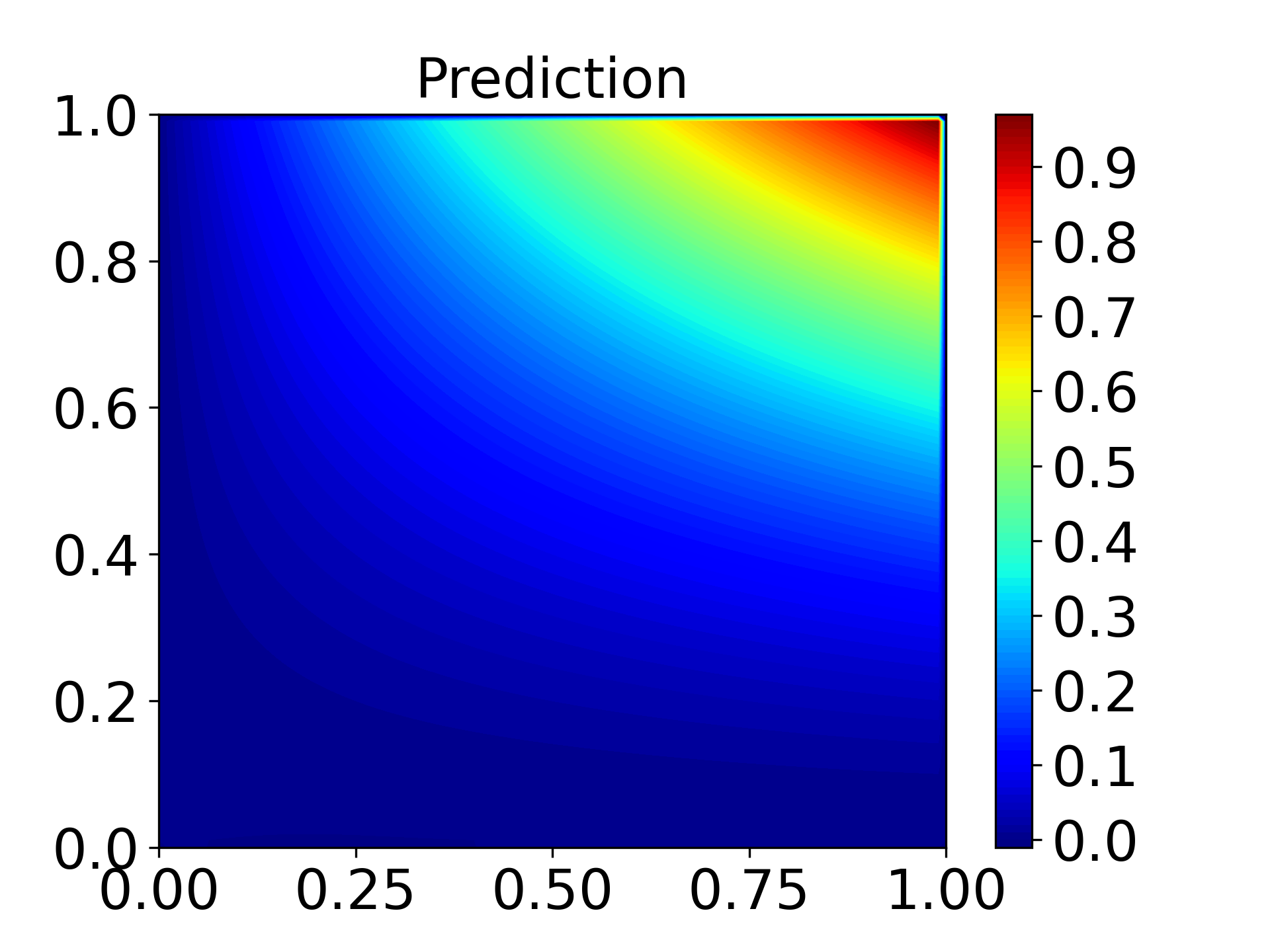}
\hfill
\includegraphics[width=0.32\textwidth]{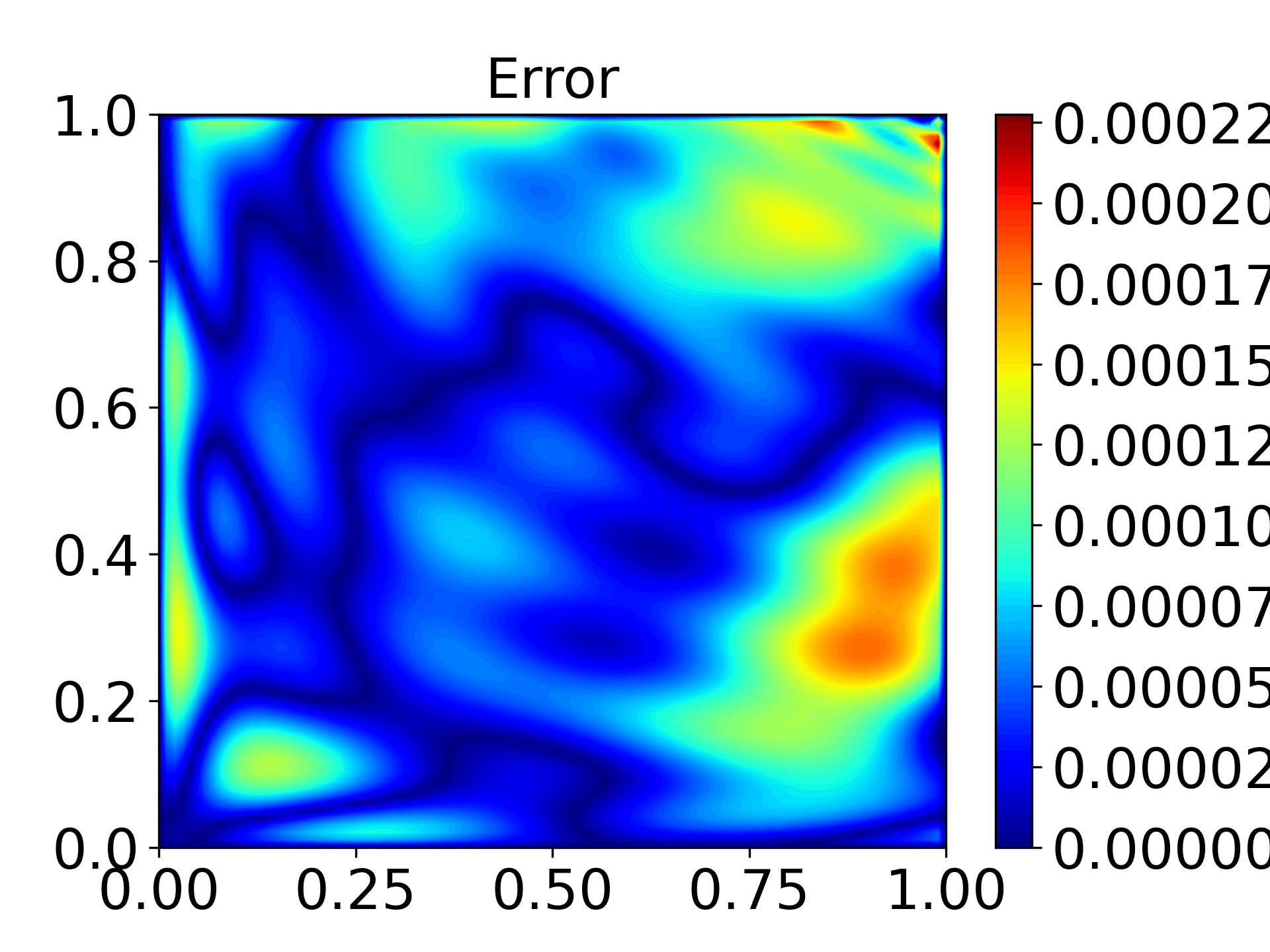}
\caption{Exact solution (left), LRX-enhanced hp-VPINN prediction with the integrated Cauchy activation (middle), and absolute error (right).}
\end{figure}

The outflow-layer problem constitutes an extremely convection-dominated benchmark with $\varepsilon=10^{-8}$, where two exponentially thin layers develop near the outflow boundaries and a corner singularity is present at the upper-right corner of the domain. Such features are notoriously difficult to approximate using standard PINN architectures.

The LRX-enhanced hp-VPINN, which incorporates the proposed integrated Cauchy representation into the improved hp-VPINN framework, consistently improves the solution accuracy under both SUPG formulations. Compared with the original improved hp-VPINN, the relative error is reduced from $1.340\times10^{-4}$ to $6.318\times10^{-5}$ for the constant-$\tau$ formulation and from $1.037\times10^{-4}$ to $3.543\times10^{-5}$ for the learned-$\tau$ formulation. Furthermore, the integrated activation substantially outperforms the original Cauchy activation under identical network architectures and training settings.

Most importantly, these improvements are obtained without modifying the stabilization strategy, loss functional, or network topology. While the original Cauchy activation already provides a slight improvement over the baseline hp-VPINN, the integrated Cauchy activation yields a substantially larger error reduction under both SUPG formulations. This indicates that the observed performance gain originates primarily from the proposed layer-resolving integrated Cauchy structure rather than from the use of a localized activation alone.

\subsection{Task 6: Circular Interior Layer Problem}
This benchmark was originally introduced by John et al.~\cite{john1997}
and has recently been revisited in \cite{matthaiou2025,frerichs2026}. Among the references considered here, the best reported result is provided by the loss-functional-oriented PINN framework in \cite{frerichs2026}.
\begin{equation*}
-\varepsilon \Delta u + 2\partial_x u +3\partial_y u+2u= f,
\qquad \mathbf{(x,y)}\in(0,1)^2,
\end{equation*}
where $\varepsilon=10^{-8}$, the exact solution has the form:
\begin{equation*}
u(x,y)=16x(1-x)(1-y)\left( \frac{1}{2}+\frac{\text{arctan}(200(0.25^2-(x-0.5)^2-(y-0.5)^2))}{\pi}\right)
\end{equation*}
Dirichlet boundary conditions are satisfied. 

\begin{table}[H]
\caption{$L^2$ error for the circular interior layers problem using different PINN-based methods.}
\centering
\footnotesize
\begin{tabular}{lcccc}
\toprule

Method
&
\\
\midrule

$\text{hp-VPINN with $L_{0.1}^{lrew}$ loss (\text{best reported result in \cite{frerichs2026}})}$
& $1.330\times 10^{-3}$ \\

$\text{Improved hp-VPINN \cite{anandh2025}}$
& $1.620\times 10^{-3}$ \\

$\text{LRX-enhanced hp-VPINN (Cauchy activation)}$
& $1.447 \times 10^{-3}$\\

$\textbf{LRX-enhanced hp-VPINN (integrated Cauchy activation)}$
& $\mathbf{8.097\times 10^{-4}}$\\

\bottomrule
\end{tabular}

\label{d2}
\end{table}

\begin{figure}[H]
\centering
\includegraphics[width=0.32\textwidth]{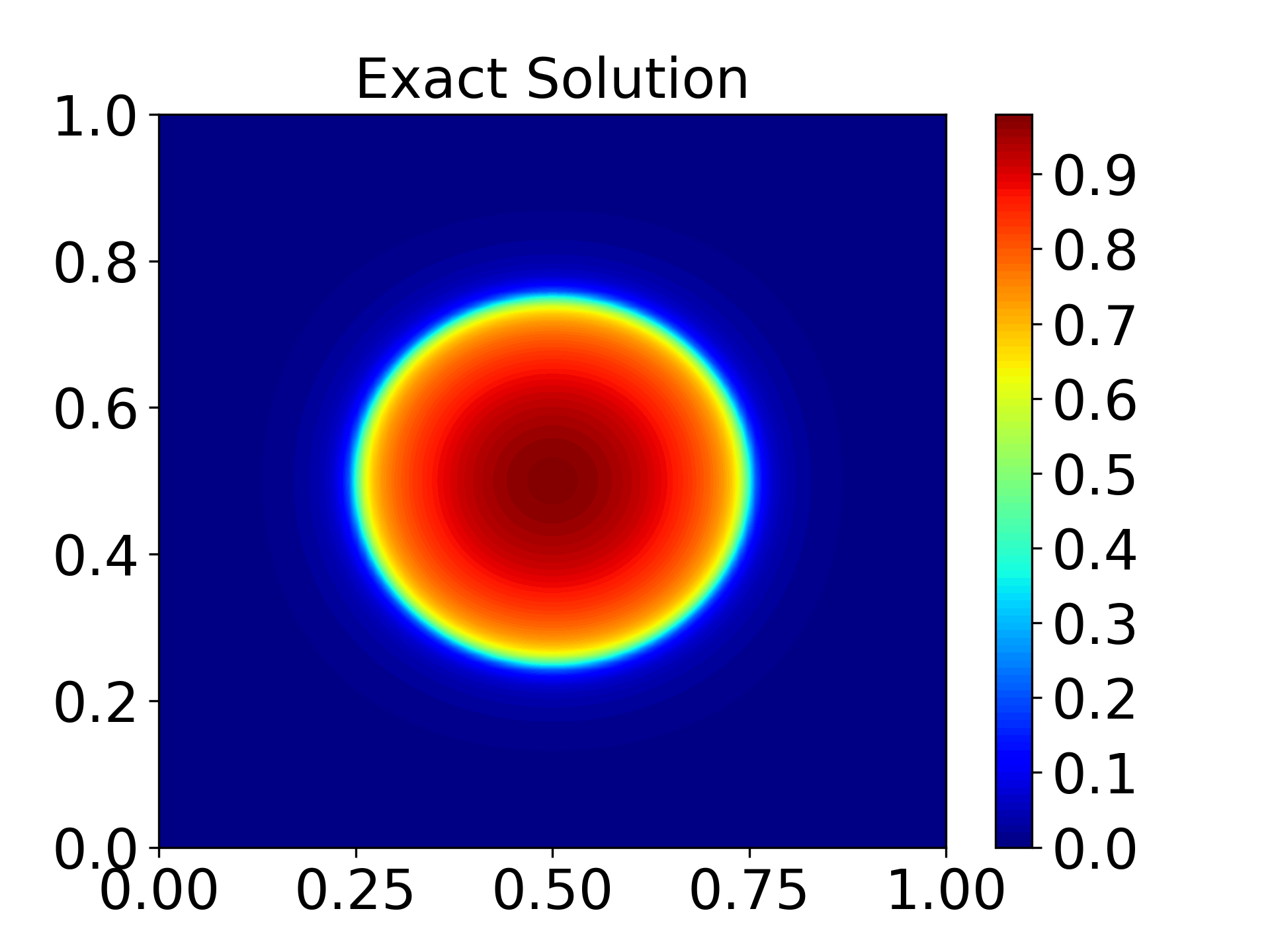}
\hfill
\includegraphics[width=0.32\textwidth]{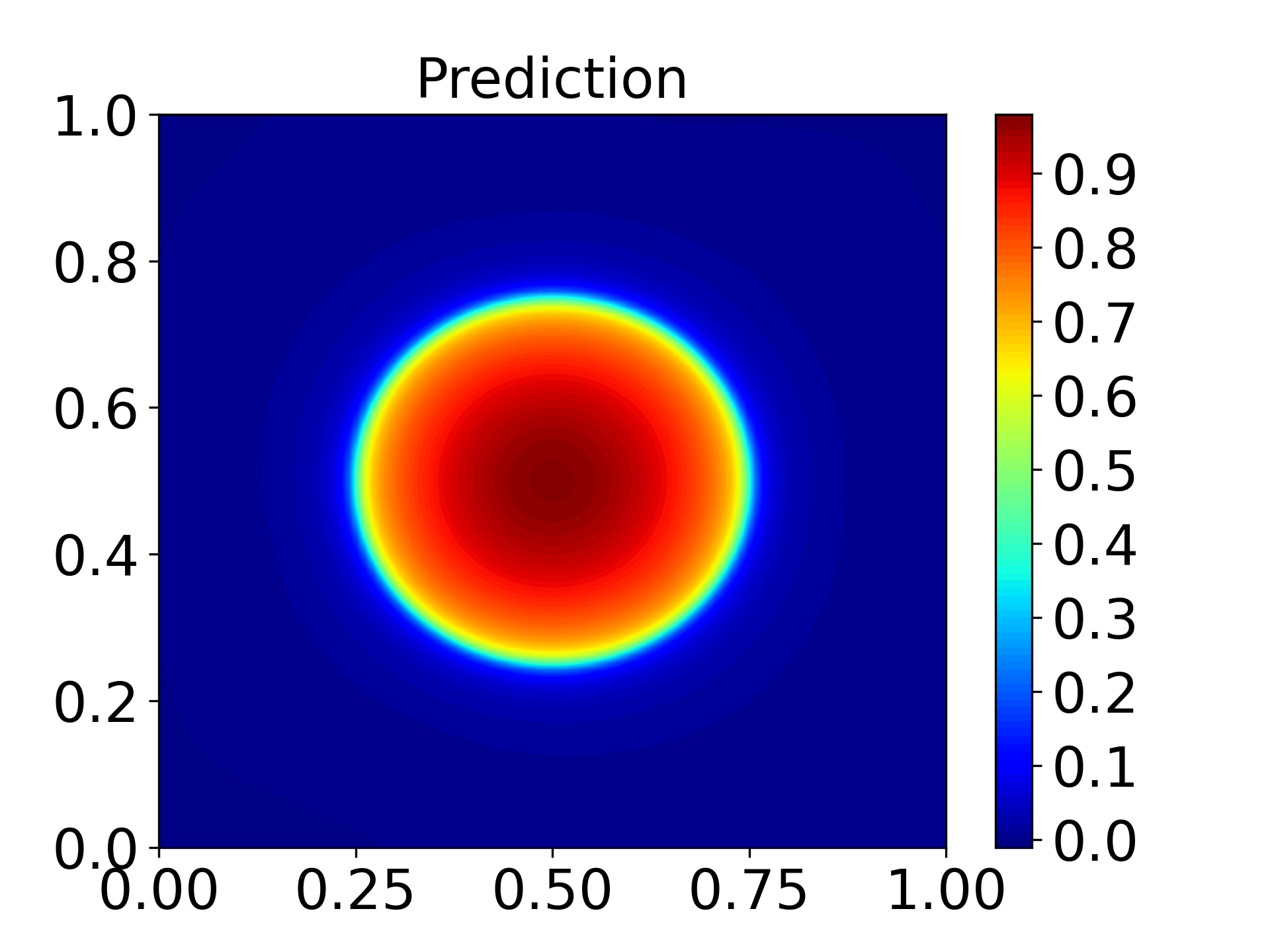}
\hfill
\includegraphics[width=0.32\textwidth]{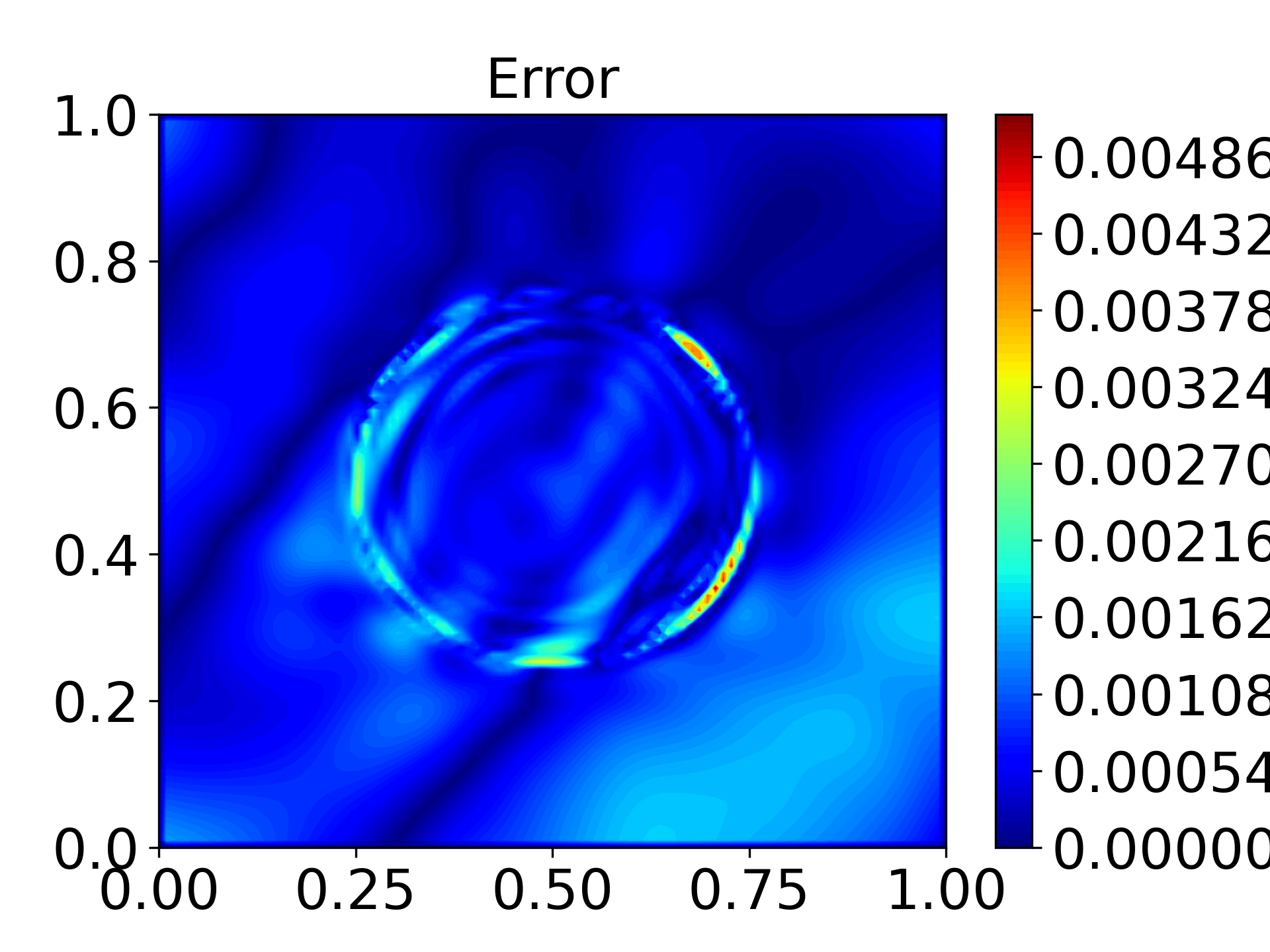}
\caption{Exact solution (left), LRX-enhanced hp-VPINN prediction with the integrated Cauchy activation (middle), and absolute error (right).}
\end{figure}

Unlike the outflow-layer benchmark, the circular interior-layer problem introduces additional geometric complexity due to the presence of a curved transition interface. Consequently, successful approximation requires not only localization but also the ability to accurately represent curved layer structures.

As shown in Table~\ref{d2}, the LRX-enhanced hp-VPINN with original Cauchy activation provides only a modest improvement over the improved hp-VPINN framework. In contrast, the LRX-enhanced hp-VPINN with proposed integrated Cauchy activation reduces the relative error from $1.447\times10^{-3}$ to $8.097\times10^{-4}$ and achieves the best overall result among all compared methods.

This observation reveals an important distinction between the original and integrated Cauchy activations. While the original activation already possesses strong localization properties, its improvement remains limited. The significantly larger gain obtained by the integrated activation demonstrates that the proposed layer-resolving transition structure, rather than localization alone, is the primary factor responsible for the observed performance improvement. Together with the results of \textbf{Task~5}, this provides strong evidence supporting the importance of the proposed integrated Cauchy activation in the overall performance improvements of the present work.

\subsection{Empirical Validation of the Layer-Resolving Representation}

Table~\ref{deps} summarizes the best-performing initial values of the scale parameter $d$ used in \textbf{Tasks 2--4}. Although these values were selected independently for different benchmark problems, a clear trend can be observed: problems with smaller diffusion parameters generally favor smaller initial scales. Since decreasing $\varepsilon$ produces increasingly narrow layer structures, this observation suggests that thinner layers require neural bases with finer representation scales.

\begin{table}[H] 
\caption{Best-performing initial values of the scale parameter $d$ used for Tasks 2--4.}
\centering
\footnotesize 
\resizebox{0.4\textwidth}{!}{ 
\begin{tabular}{lcccc} 
\toprule & Task 2 & Task 3 & Task 4\\ \midrule 
$\varepsilon=0.5$ & 0.5 & 0.3 & 0.5\\ 
$\varepsilon=0.1$ & 0.2 & 0.25 & 0.35\\ 
$\varepsilon=0.05$ & 0.18 & 0.2 & 0.2 \\ 
$\varepsilon=0.01$ & 0.02 & 0.06 & 0.09 \\ 
\bottomrule 
\end{tabular} 
} 

\label{deps} 
\end{table}

To further investigate this relationship, we perform an ablation study by varying the initial scale parameter $d$ and selecting the best-performing value after 5000 Adam iterations. The results are shown in Figure~\ref{initial}. Across different benchmark problems, the optimal initial scale parameter consistently decreases as $\varepsilon$ becomes smaller. This behavior agrees with the neural resolution principle established in Section~2, which predicts that the representation scale should decrease together with the physical layer width. Consequently, thinner layers are more effectively approximated by integrated Cauchy bases with smaller initial scales.

\begin{figure}[H]
\centering
\includegraphics[width=0.8\textwidth]{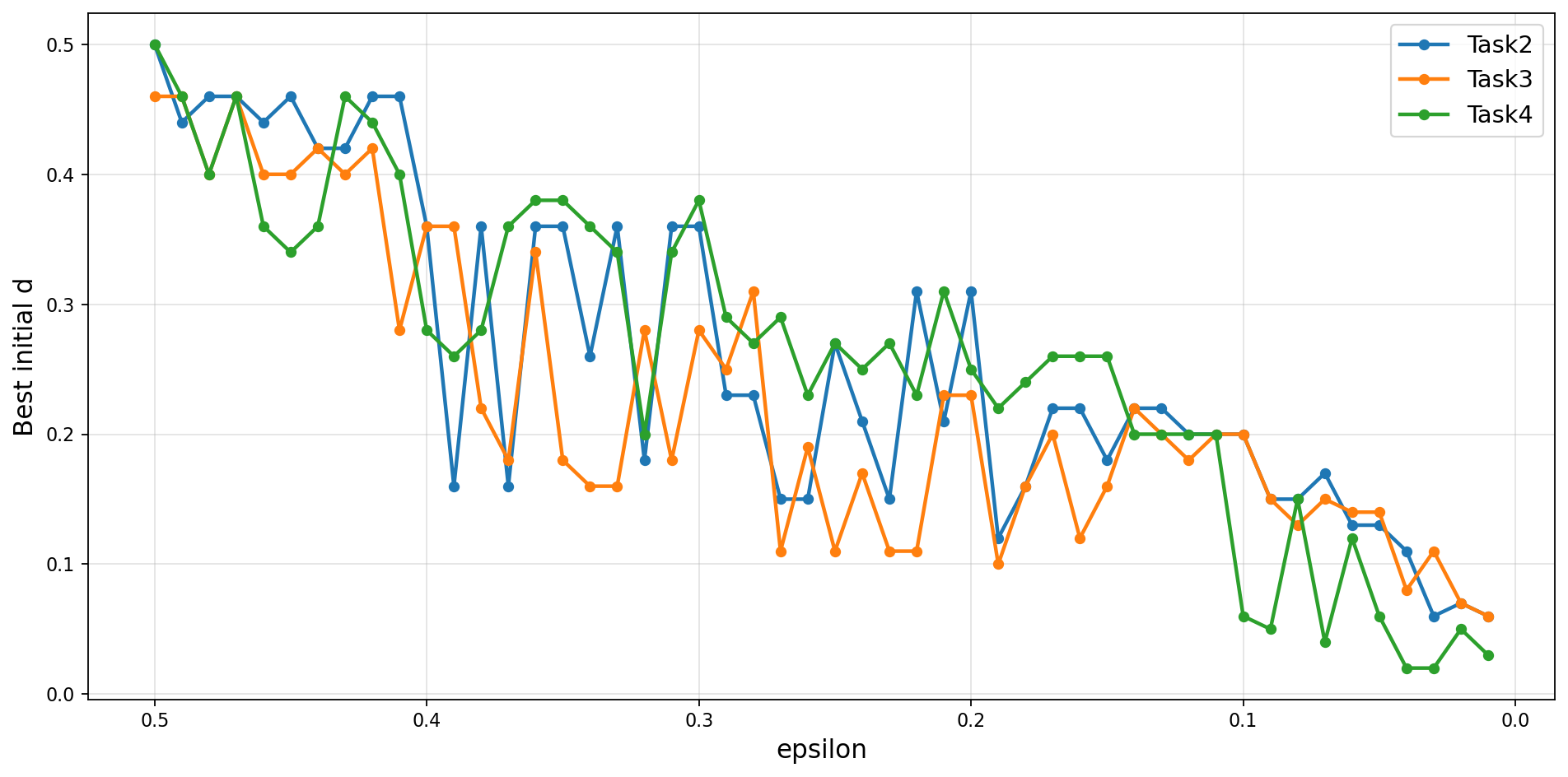}
\caption{Optimal initial scale parameter $d$ obtained from the ablation study under different parameters $\varepsilon$.}
\label{initial}
\end{figure}

We further evaluate robustness with respect to random initialization by repeating each experiment using six different random seeds 0-5. Table~\ref{seed} reports the mean $L^2$ error together with the maximum deviation from the mean. Across all benchmark problems, the proposed method achieves low errors with moderate variations across different initializations, indicating stable and reproducible performance. These results suggest that the observed accuracy improvements are not due to favorable random seeds or isolated training runs, but instead arise from the proposed layer-resolving representation itself.

\begin{table}[H]
\centering
\caption{Robustness evaluation over random seeds 0--5. Results are reported as mean $L^2$ error $\pm$ maximum deviation from the mean.}
\label{seed}
\begin{tabular}{lc}
\toprule
Tasks & $L^2$ Error with seeds 0-5\\
\midrule
Task 1 ($\varepsilon$=1)&  $7.677 \times 10^{-5} \pm 1.622 \times 10^{-5}$ \\
Task 2 ($\varepsilon$=0.5)& $ 5.862 \times 10^{-5} \pm 1.704 \times 10^{-5}$ \\
Task 2 ($\varepsilon$=0.1) & $ 5.945 \times 10^{-5} \pm 1.804 \times 10^{-5}$ \\
Task 2 ($\varepsilon$=0.05) & $ 2.239 \times 10^{-4} \pm 1.129 \times 10^{-4}$ \\
Task 2 ($\varepsilon$=0.01) & $ 2.889 \times 10^{-4} \pm 1.044 \times 10^{-4}$\\
Task 3 ($\varepsilon$=0.5) & $ 3.907 \times 10^{-5} \pm 1.191 \times 10^{-5}$\\
Task 3 ($\varepsilon$=0.1) & $ 2.650 \times 10^{-5} \pm 0.632 \times 10^{-5}$\\
Task 3 ($\varepsilon$=0.05) & $ 5.138 \times 10^{-5} \pm 1.682 \times 10^{-5}$\\
Task 3 ($\varepsilon$=0.01) & $ 5.566 \times 10^{-3} \pm 0.922 \times 10^{-3}$\\
Task 4 ($\varepsilon$=0.5) & $ 3.345 \times 10^{-5} \pm 0.386 \times 10^{-5}$\\
Task 4 ($\varepsilon$=0.1) & $ 5.105 \times 10^{-5} \pm 1.002 \times 10^{-5}$\\
Task 4 ($\varepsilon$=0.05) & $ 2.552 \times 10^{-4} \pm 0.537 \times 10^{-4}$\\
Task 4 ($\varepsilon$=0.01) & $ 8.855 \times 10^{-3} \pm 1.165 \times 10^{-3}$\\
Task 5 ($\varepsilon=10^{-8}$) & $ 3.335 \times 10^{-5} \pm 0.635 \times 10^{-5}$\\
Task 6 ($\varepsilon=10^{-8}$) & $ 7.254 \times 10^{-4} \pm 1.723 \times 10^{-4}$\\
\bottomrule
\end{tabular}
\end{table}

An additional practical advantage of the proposed LRX-PINN comes from the explicit analytic form of the integrated Cauchy basis. For a single-hidden-layer
Cauchy network, the function values and the derivatives entering the PDE
residual can be evaluated either by automatic differentiation or by the
closed-form identities derived in Section~2. These two implementations
represent the same trial space and the same residual formulation; therefore,
the closed-form implementation should be understood as a computational
acceleration rather than as a separate approximation mechanism.

Table \ref{time} reports the training time obtained with automatic differentiation and with the closed-form implementation. The closed-form version substantially reduces the computational overhead of LRX-PINN. Since the underlying approximation space and loss functional are unchanged, this replacement should not be interpreted as a separate accuracy-improving mechanism. Instead, the main benefit is that the explicit derivative structure of the integrated Cauchy basis can be exploited to accelerate residual and gradient evaluation.

\begin{table}[H] 
\caption{Training time of different PINNs.}
\centering
\footnotesize  
\begin{tabular}{lcccc} 
\toprule & PIKAN & Fourier PINN  & LRX-PINN & LRX-PINN with explicit gradient\\ \midrule 
Task 1  & 4041.87s & 441.39s & 1120.38s & 202.12s\\ 
Task 2 ($\varepsilon=0.5$) & 7856.24 s & 496.41 s & 903.45 s & 281.35s\\ 
Task 2 ($\varepsilon=0.1$) & 5651.36 s & 564.05 s & 1057.29s & 359.77s\\ 
Task 2 ($\varepsilon=0.05$) & 6451.36 s & 511.10 s  & 1158.20 s & 452.68s\\ 
Task 2 ($\varepsilon=0.01$) & 5647.80s  & 673.51s & 1168.91s  & 506.82s\\ 
Task 3 ($\varepsilon=0.5$) & 7705.79s & 555.64 s & 651.26s & 430.52s\\ 
Task 3 ($\varepsilon=0.1$) & 7874.28s  & 509.02s & 1046.17s & 434.38s\\ 
Task 3 ($\varepsilon=0.05$) & 8277.91s & 566.63s & 1826.48s & 499.21s\\ 
Task 3 ($\varepsilon=0.01$) & 8524.07 s & 576.49 s & 1232.77 s & 504.26s\\ 
Task 4 ($\varepsilon=0.5$) & 8890.24s & 507.84s  & 1160.69s & 424.62s\\ 
Task 4 ($\varepsilon=0.1$) & 9910.87 s & 535.06 s & 1123.81s & 531.45s\\ 
Task 4 ($\varepsilon=0.05$) & 8059.20 s & 532.76 s  & 1380.25s & 499.77s\\ 
Task 4 ($\varepsilon=0.01$) & 11285.41 s & 576.49 s & 1338.37 s & 517.83s \\ 
\bottomrule 
\end{tabular} \label{time}
\end{table}

The timing results show that the analytic implementation can make the compact LRX-PINN significantly faster in practice under the tested implementation. This acceleration is a consequence of the explicit structure of the single-layer integrated Cauchy representation: the residual terms involve only \(\Phi_d\), \(\phi_d\), and
\(\phi_d'\), which can be evaluated directly. We emphasize that this implementation does not change the model class, the loss functional, or the training objective. Hence the accuracy gains reported in the preceding sections should be attributed to the layer-resolving representation itself, whereas the closed-form implementation provides an additional computational benefit.

\section{Conclusion}\label{conclusion}

This work investigates convection-dominated convection--diffusion problems from the perspective of neural representation design. Motivated by the value--derivative structure commonly observed in thin-layer solutions, we develop a Layer-Resolving XNet Physics-Informed Neural Network (LRX-PINN) based on integrated Cauchy representations. By integrating localized Cauchy rational kernels, the proposed basis exhibits a transition-type structure at the solution level while recovering localized rational features after differentiation. This design aligns the neural trial space with the characteristic structure of convection-dominated layers, whose values remain bounded while their derivatives are sharply localized.

We establish a representation-level analysis based on layer scaling, approximation inheritance for analytic layer profiles, and residual-structure considerations. The analysis suggests that integrated Cauchy bases inherit the exponential approximation mechanism of Cauchy/XNet representations, while their derivatives provide localized Cauchy-type features in the convection--diffusion residual. The effective neural scale $d/\|w\|$ provides a natural resolution parameter for thin-layer structures.

Numerical experiments on interior-layer, boundary-layer, mixed-layer, outflow-layer, corner-layer, and curved-layer benchmarks show that LRX-PINN consistently improves accuracy, robustness, and parameter efficiency compared with existing PINN approaches. The improvement becomes more pronounced in strongly convection-dominated regimes where extremely thin layers are present. Additional experiments with hp-VPINN-based frameworks further indicate that the performance gains are primarily associated with the proposed layer-resolving representation, rather than with a specific loss functional, stabilization strategy, or optimization setting.

A key outcome of this study is the parameter efficiency of the proposed representation. In Tasks 1--4, LRX-PINN uses 841 trainable parameters, whereas PIKAN and Fourier PINN use 7065 and 3201 parameters, respectively. Thus, LRX-PINN uses only approximately $11.9\%$ and $26.3\%$ of the parameters of PIKAN and Fourier PINN, while achieving lower errors across all reported benchmarks. This suggests that matching the neural basis to the layer structure can be more effective than simply increasing the number of trainable parameters.

More broadly, this work suggests that improving physics-informed neural networks may not always require deeper architectures, more complex optimization strategies, or increasingly sophisticated loss functionals. Instead, designing neural representations according to the intrinsic analytical structure of PDE solutions provides a complementary and often effective approach for multiscale problems. The proposed LRX-PINN illustrates how incorporating layer-aware structures into the neural trial space can simultaneously improve approximation accuracy, parameter efficiency, and computational efficiency. We expect that this representation-oriented perspective may be extended to other PDEs involving localized structures, sharp interfaces, and multiscale phenomena. More generally, this work indicates that aligning neural representations with the mathematical structure of PDEs is a useful principle for designing efficient physics-informed computational methods.

The present work focuses on steady linear convection--diffusion problems in the singularly perturbed regime. Before addressing nonlinear convection-dominated systems such as Burgers or Navier--Stokes equations, it is important to first address the representation mismatch caused by extremely thin layers in the linear setting. This provides a useful foundation for developing robust physics-informed methods for more complex transport-dominated problems. Future work will further extend the proposed framework to nonlinear systems, higher-dimensional problems, and adaptive selection of representation scales.

\section*{Data Availability Statement}

The data that support the findings of this study are available from
the corresponding author upon reasonable request.

\section*{Declaration of Interests}

The authors report no conflict of interest.

\end{document}